\documentclass[11pt]{article}

\usepackage[T1]{fontenc}
\usepackage[margin=1in]{geometry}
\usepackage{authblk}
\usepackage{graphicx}
\usepackage{amssymb}
\usepackage{amsthm}
\usepackage{amsfonts}
\usepackage{mathrsfs}
\usepackage{mathtools}
\usepackage{xcolor}
\usepackage{subcaption}
\usepackage{multirow}
\usepackage{float}
\usepackage{booktabs}
\usepackage{enumitem}
\usepackage{microtype}
\usepackage{comment}
\setlist[enumerate,1]{label=(\arabic*),ref=\arabic*}
\setlist[enumerate,2]{label=(\theenumi\alph*),ref=\theenumi\alph*}
\usepackage{hyperref}
\hypersetup{%
    colorlinks=true,
    citecolor=darkgreen,
    linkcolor=blue,
    urlcolor=blue,
}
\usepackage[capitalise]{cleveref}
\newcommand{\email}[1]{\href{mailto:#1}{#1}}
\newcommand{\T}{\mathbb{T}}
\newcommand{\Z}{\mathbb{Z}}
\newcommand{\E}{\mathbb{E}}
\newcommand{\R}{\mathbb{R}}

\newcommand{\N}{\mathbb{N}}

\newcommand{\e}{\mathrm{e}}

\newcommand{\ind}{\mathbf{1}}
\newcommand{\X}{\mathcal{X}}
\renewcommand{\L}{\mathcal{L}}
\renewcommand{\S}{\mathscr{S}}
\newcommand{\bigO}{\mathrm{O}}
\renewcommand{\geq}{\geqslant}
\renewcommand{\leq}{\leqslant}
\renewcommand{\t}{\mathsf T}

\renewcommand{\Im}{\operatorname{Im}}
\newcommand\sn[2]{{#1}\times 10^{#2}}

\DeclareMathOperator{\Div}{div}
\DeclareMathOperator{\Var}{Var}

\DeclarePairedDelimiter\ip{\langle}{\rangle}
\DeclarePairedDelimiter\abs{\lvert}{\rvert}
\DeclarePairedDelimiter\norm{\lVert}{\rVert}
\DeclarePairedDelimiter\paren{\lparen}{\rparen}
\DeclarePairedDelimiter\bkt{\lbrack}{\rbrack}

\definecolor{darkgreen}{rgb}{0.0, 0.5, 0.0}

\newtheorem{theorem}{Theorem}[section]
\newtheorem{corollary}[theorem]{Corollary}
\newtheorem{assumption}{Assumption}

\newtheorem{remark}[theorem]{Remark}
\newtheorem{lemma}[theorem]{Lemma}
\newtheorem{proposition}[theorem]{Proposition}

\theoremstyle{definition}

\crefname{assumption}{Assumption}{Assumptions}
\crefname{figure}{Figure}{Figures}

\let\oldparagraph=\paragraph
\renewcommand\paragraph[1]{\oldparagraph{#1.}}

\newcommand{\wcL}{\widetilde{\mathcal{L}}}
\newcommand{\G}{\mathcal{G}}
\newcommand{\Id}{\mathrm{Id}} 
\newcommand{\TestTTCF}{\estTmp^{T,K,\eta}_{\rm TTCF}}
\newcommand{\TestNEMD}{\estTmp^{T,\eta}_{\rm NEMD}}
\newcommand{\TestTTCFcenter}{\estTmp^{T,K,\eta}_{\rm center}}

\usepackage[%
	backend=biber,
	style=trad-abbrv,
	url=false,
	doi=false,
	isbn=false
]{biblatex}
\AtEveryBibitem{\clearfield{month}}
\AtEveryCitekey{\clearfield{month}}
\AtEveryBibitem{\clearlist{location}}
\AtEveryCitekey{\clearlist{location}}
\DeclareFieldFormat{volume}{volume \textbf{#1}}
\DeclareFieldFormat[article]{volume}{\textbf{#1}}
\newcommand{\psip}{\widetilde{\mu}_\eta}
\newcommand{\bLip}{\|b\|_{\rm Lip}}
\newcommand{\RLip}{\|R\|_{\rm Lip}}
\newcommand{\mucoup}{\mu_{\rm coup}}
\newcommand{\Binfz}{B_{n,0}^{\infty}}

\newcommand{\GKest}{\mathcal{\widehat{\rho}}^{T,K}_{\rm GK}}
\newcommand{\estTmp}{\widehat{\rho}}
\newcommand{\Test}{\estTmp^{T,K,\eta}_{\rm trans}} 
\newcommand{\TSest}{\estTmp^{T,K,\eta}_{\rm sub}} 
\newcommand{\asTSest}{\estTmp^{T,\eta,\alpha}_{\rm sub}} 
\newcommand{\aTSest}{\estTmp^{T,K,\eta,\alpha}_{\rm sub}} 

\title{Transient subtraction: A control variate method for computing transport coefficients}
\author[1]{Pierre Monmarch\'e$^{a,}$}
\author[2,3]{Renato Spacek$^{b,}$}
\author[3,2]{Gabriel Stoltz$^{c,}$}
\affil[ ]{\footnotesize $^a$\email{pierre.monmarche@sorbonne-universite.fr},
                        $^b$\email{renato.spacek@enpc.fr},
                        $^c$\email{gabriel.stoltz@enpc.fr}}
\affil[1]{\footnotesize LJLL and LCT, Sorbonne Universit\'e, Paris, France}
\affil[2]{\footnotesize MATHERIALS team, Inria Paris, France}
\affil[3]{\footnotesize CERMICS, \'Ecole des Ponts, France}

\begin{document}
\maketitle

\begin{abstract}
In molecular dynamics, transport coefficients measure the sensitivity of the invariant probability measure of the stochastic dynamics at hand with respect to some perturbation. They are typically computed using either the linear response of nonequilibrium dynamics, or the Green--Kubo formula. The estimators for both approaches have large variances, which motivates the study of variance reduction techniques for computing transport coefficients. We present an alternative approach, called the \emph{transient subtraction technique} (inspired by early work by Ciccotti and Jaccucci in 1975), which amounts to simulating a transient dynamics started off equilibrium and relaxing towards the equilibrium state, from which we subtract a sensibly coupled equilibrium trajectory, resulting in an estimator with smaller variance. We present the mathematical formulation of the transient subtraction technique, give error estimates on the bias and variance of the associated estimator, and demonstrate the relevance of the method through numerical illustrations for various systems.
\end{abstract}

\section{Introduction}
\label{sec:introduction}
When considering large systems of interacting particles, quantities of interest are typically macroscopic properties, such as temperature and pressure, rather than microscopic ones. Generally, full microscopic descriptions are not only too large to be reasonably considered, but also largely uninteresting. From a numerical viewpoint, molecular dynamics provides an effective way of bridging the microscopic and macroscopic properties of such systems through computer simulations; see \cite{tuckerman2010, leimkuhler2015,allen2017} for reference textbooks. These simulations are typically done via the numerical realization of a stochastic differential equation (SDE), such as the Langevin dynamics, which evolves the positions~$q$ and momenta~$p$ as
\begin{equation}
\begin{aligned}
\begin{cases}
    dq_t = M^{-1} p_t \, dt, \\
    dp_t = -\nabla V(q_t) \, dt - \gamma M^{-1} p_t \, dt + \sqrt{\dfrac{2\gamma}{\beta}} \, dW_t,
    \label{eq:lang_dynamics}
\end{cases}
\end{aligned}
\end{equation}
where $V$ is the potential energy function, $M$ the mass matrix, $\gamma>0$ the damping coefficient, $\beta>0$ the inverse temperature (up to a factor~$k_\mathrm{B}$, with~$k_\mathrm{B}$ the Boltzmann constant) and~$W_t$ a standard multidimensional Brownian motion.

One particular application of molecular dynamics is the computation of transport coefficients (such as the diffusivity, mobility and shear viscosity), which encode important physical properties of materials, and in particular measure how quickly a perturbed system returns to steady-state. At the microscopic level, transport coefficients are defined as the proportionality constant between the magnitude $\eta\ll 1$ of some external forcing exerted on the system, and some flux induced by this forcing. The flux is represented as the steady-state average~$\E_\eta(R)$ for some given observable~$R$ with average 0 with respect to the equilibrium system $(\eta=0)$. This can be made precise through the framework of \emph{linear response theory}; see \cite[Chapter 8]{chandler1987} for an introduction, as well as~\cite{evans2016} for a more comprehensive treatment of nonequilibrium systems. To numerically realize this, one considers a nonequilibrium system by adding a perturbation of magnitude $\eta$ to the reference dynamics at hand (e.g., Langevin dynamics), and the appropriate flux is then measured as a time-average over a long trajectory; this is known as the nonequilibrium molecular dynamics (NEMD) method \cite{ciccotti2005}.

Alternatively, the linear response can be reformulated as an equilibrium integrated correlation, known as the \emph{Green--Kubo} (GK) formula \cite{green1954,kubo1957}. Both the NEMD and Green--Kubo methods are commonly used, and each has their advantages and drawbacks; see \cite{stoltz2024} for a detailed comparison of both approaches. 

Although less common, a third class of techniques consists of methods based on transient dynamics, which typically rely on monitoring the system's relaxation to steady-state after an initial perturbation (unlike the NEMD and GK methods, which are based on steady-state averages). As will be made precise in \cref{subsec:num_tech}, transient methods can be applied in two different ways: (i) starting from an equilibrium system with perturbed initial conditions, and allowing the system to relax back to its equilibrium steady-state, e.g., the momentum impulse relaxation \cite{arya2000} and the approach-to-equilibrium molecular dynamics methods \cite{lampin2013}; or a somewhat dual approach, carried out by (ii) applying a driving force to an equilibrium system and monitoring its relaxation towards a nonequilibrium steady-state, such as the transient-time correlation function (TTCF) method~\cite{morriss1987,evans1988}. This method provides an unbiased estimator of the full nonlinear response whatever the magnitude of the forcing (in the limit of infinite integration times), in contrast to the class of techniques based on relaxation to equilibrium. It was used to study various systems under realistic (\emph{i.e.} small enough) forcings, see for instance~\cite{delhommelle2005, pan2006, bernardi2012, bernardi2016, maffioli2022, maffioli2024} for some works. A more precise discussion is provided in Sections~\ref{subsec:TTCF} and \ref{sec:error_TTCF}.

All three classes of methods suffer from severe numerical difficulties, in particular because of large statistical errors, as made precise in \cref{subsec:num_tech}. For NEMD, the statistical error mainly arises from the large signal-to-noise ratio (due to the small magnitude of the perturbation $\eta$), which requires long integration times to offset the variance. For Green--Kubo, the statistical error scales linearly with the integration time $T$, while correlations decay in time, so it amounts to integrating a small quantity plagued by a large statistical error \cite{sousaoliveira2017}. Similar estimates are obtained for the TTCF method, see the discussion in \cref{sec:error_TTCF}. 

There have been several attempts at more efficient methods to compute transport coefficients in the context of variance reduction \cite{pavliotis2023,spacek2023,blassel2024}. In particular, one such method is known as the \emph{subtraction technique}, developed and investigated in \cite{ciccotti1975}, and further explored in \cite{ciccotti1979}, in the context of transient methods. The method is based on computing the difference between two trajectories, one started at equilibrium and one started slightly off equilibrium and relaxing to equilibrium. As discussed in both works~\cite{ciccotti1975,ciccotti1979} (which consider a deterministic framework), the high correlation between trajectories is a natural artifact of the deterministic dynamics for reasonably short integration times. This allows for the statistical error to be effectively subtracted out through the equilibrium trajectory, thus making the subtraction step an effective control variate. In stochastic settings, however, using independent noises for equilibrium and nonequilibrium trajectories (corresponding to $\eta=0$ and $\eta\ne 0$, respectively) results in uncorrelated trajectories. This suggests the need for constructing a sensible coupling between the two systems; otherwise, the subtraction step would essentially amount to adding two independent random variables, doubling the variance of the estimator at hand. 

One way to overcome this issue is to consider couplings, which have been used as a control variate to compute transport coefficients \cite{goodman2009,garnier2022}. One particular example of a common coupling strategy is synchronous coupling, which amounts to using the same noise for both dynamics. A major challenge with coupling techniques, however, is ensuring that trajectories stay coupled for long times. This is especially problematic for systems which rely on long-time averages for convergence, e.g., NEMD \cite{darshan2024}. Typically, one hopes to obtain convergence of time-averages before trajectories decouple, but this cannot be assumed in general without additional (and often restrictive) requirements.

Synchronous coupling, for instance, typically requires conditions such as global dissipativity in order to ensure long-time couplings of the trajectories. In general, however, global dissipativity is only obtained under strong conditions. One such example is when overdamped Langevin dynamics $dq_t = -\nabla V(q_t) \, dt + \sqrt{2\beta^{-1}} \, dW_t$ are considered, for some strongly convex potential $V$, which is too restrictive a requirement for actual applications in MD. This suggests that synchronous coupling is typically impractical for estimators that require long-time integration such as NEMD, as the decoupling time is much shorter than the time needed for convergence with no global dissipativity. Let us however mention that, in some cases, for instance at high temperatures, synchronously coupled trajectories might not decouple at all even with no global dissipativity; see \cite{monmarche2023}. A natural way to address this problem would be to construct couplings with milder conditions which guarantee long-time couplings but this remains challenging (see for instance~\cite{darshan2024}).

We adopt in this work an alternative viewpoint: We consider methods for which convergence of an observable is feasible over short-times. In particular, we devise a transient method, consisting of an initially perturbed trajectory relaxing to equilibrium. We thus do not require long-time averages for convergence, which suggests that we can use synchronous coupling under weak conditions, provided that the relaxation time is smaller than the decoupling time. Indeed, even though the dynamics might start to decouple before relaxation, the variance of the estimator for the transport coefficient might nonetheless be decreased due to the control variate.

For systems with longer relaxations times, it may be beneficial to turn to another transient method, namely TTCF; see the discussion in Sections~\ref{sec:error_TTCF} and~\ref{sec:conclusion}. It would be interesting to compare both approaches on test cases and realistic systems, but this is outside the scope of this work.

\paragraph{Outline} This work is organized as follows. We discuss in \cref{sec:trans_coeff} some standard numerical methods for approximating transport coefficients, and present approaches based on integrating dynamics in the transient regime. Then, by employing the subtraction technique to the transient method starting off equilibrium and relaxing to equilibrium, we construct in \cref{sec:transient} an improved transient subtraction estimator. We provide some error analysis on its bias and variance. We then illustrate the efficacy of our method with numerical results for several systems in \cref{sec:numerical_lang}, namely by computing the mobility for one-dimensional Langevin dynamics, and mobility and shear viscosity for a Lennard--Jones fluid. Finally, conclusions and extensions are discussed in \cref{sec:conclusion}.

\section{Transient method to compute transport coefficients}
\label{sec:trans_coeff}
We discuss in this section the definition and computation of transport coefficients, and in particular the use of a transient method for their approximation. We start by presenting in \cref{subsec:gen_setting} the setting used to compute transport coefficients for a general SDE, then overview their standard numerical approximations and associated numerical difficulties in \cref{subsec:num_tech}. We then introduce the transient method we consider in this work in \cref{subsec:transient}, and also the dual approach, namely TTCF, in \cref{subsec:TTCF}. 

\subsection{General setting}
\label{subsec:gen_setting}
Consider a general time-homogeneous SDE with additive noise defined on the state-space $\mathcal{X}$, where~$\mathcal{X}$ is typically $\R^d$ or $\T^d$ (with $\T = \R/\Z$ the one-dimensional torus):
\begin{equation}
    dX_t = b(X_t) \, dt + \sigma \, dW_t,
    \label{eq:general_SDE}
\end{equation}
where $b\colon \mathcal{X} \to \R^d$ is a smooth function, $\sigma \in \R^{d\times m}$ is a constant matrix and $W_t$ is a standard~$m$-dimensional Brownian motion. We assume that~\eqref{eq:general_SDE} admits a unique strong solution (which is the case for instance when $b$ is globally Lipschitz). We restrict ourselves to cases where $\sigma$ is constant, as the dynamics of interest considered later on, namely Langevin dynamics, only involve additive noise, and also because the coupling method we introduce in \cref{subsec:constructing_method} is considerably easier to formulate in this setting. The dynamics \eqref{eq:general_SDE} has associated infinitesimal generator
\begin{equation}
    \notag
    \L = b^\t\nabla + \frac{1}{2}\sigma\sigma^\t \colon \nabla^2 = \sum_{i=1}^d b_i\partial_{x_i} + \frac{1}{2}\sum_{i,j=1}^d\sum_{k=1}^m \sigma_{ik}\sigma_{jk} \partial_{x_ix_j}^2,
\end{equation}
where $\colon$ denotes the Frobenius inner product. Throughout this work, we assume that~\eqref{eq:general_SDE} admits a unique invariant probability measure $\mu$ with a positive density with respect to the Lebesgue measure. 
We denote by 
\[
L^2_0(\mu) = \left\{\varphi \in L^2(\mu) \, \middle| \, \int_\X \varphi \, d\mu = 0\right\}
\]
the space of~$L^2(\mu)$ functions with average~0 with respect to~$\mu$.

Transport coefficients measure how the steady state of the reference dynamics~\eqref{eq:general_SDE} changes when some external forcing is applied to it. This external forcing typically arises as an extra drift term of magnitude~$\eta$, with~$|\eta|$ small in order for the forcing to be considered as a small perturbation. In this context, the transport coefficient~$\rho$ is defined as the proportionality constant between the steady-state flux of some observable $R$ of interest, and the magnitude of the external forcing needed to induce it, known as the \emph{linear response}; see \cite[Chapter 8]{chandler1987} for an introduction to linear response theory, and for instance \cite[Section 2]{spacek2023} for a synthetic presentation. We assume that the observable~$R$ has average~0 with respect to $\mu$ (without loss of generality, as it can always be recentered in case it has a nonzero average). The linear response can be reformulated in terms of an integrated time-correlation function, known as the Green--Kubo formula. For simplicity, we do not further recall the framework of linear response theory and instead directly write the Green--Kubo formula:
\begin{align}
  \rho = \int_0^{+\infty} \E_\mu(R(X_t) S(X_0)) \, dt,\label{eq:gk}
\end{align}
where $S \in L^2_0(\mu)$ is the conjugate response function, which depends on the extra drift term added to perturb the dynamics (see \cite[Section 5.2.3]{lelievre2016} for a precise definition, and~\eqref{eq:conjugate_response} below), and where the expectation $\E_\mu$ is taken with respect to all initial conditions~$X_0\sim \mu$, and over all realizations of the dynamics~\eqref{eq:general_SDE}. Let us emphasize that the conjugate response $S$ has average 0 by construction.

\subsection{Numerical techniques to compute transport coefficients}
\label{subsec:num_tech}
Transport coefficients can be numerically estimated using a variety of techniques. Generally, such techniques fall into one of three main categories (see \cite{lelievre2016,stoltz2024} for a detailed discussion and elements of numerical analysis):

\begin{enumerate}
    \item {\bf Equilibrium techniques based on the Green--Kubo formula \eqref{eq:gk}.} In order to numerically realize \eqref{eq:gk}, one constructs an estimator by (i) truncating the time-integral to finite integration time $T$; and (ii) approximating the expectation with an average over $K$ independent trajectories of the system $(X_t^k)_{t\geq 0}$ with $1\leq k\leq K$. This leads to the following natural estimator:
    \begin{equation}
        \GKest = \frac{1}{K}\sum_{k=1}^K \int_0^T R(X_t^k)S(X_0^k) \, dt.
        \label{eq:GK_est}
    \end{equation}
The sources of error associated with the estimator \eqref{eq:GK_est} are

\begin{itemize}
    \item A statistical error $\bigO(T)$, which scales linearly with the time lag \cite{sousaoliveira2017,plechac2022, pavliotis2024,gastaldello2024} and is typically the largest source of error;
    \item A time truncation bias, which is small as correlations are typically exponentially decaying (as discussed for instance in~\cite{plechac2022});
    \item A discretization bias, which arises from the finiteness of the timestep used to discretize~\eqref{eq:general_SDE}, and from quadrature formulas for the time integral \cite{leimkuhler2016,lelievre2016}.
\end{itemize}
The various sources of error suggest carefully choosing $T$ in order to minimize the error as a tradeoff between $T$ large enough for the time truncation bias to be small, but not too large in order to limit the increase in variance.
    
    \item {\bf Nonequilibrium steady-state techniques.} This method is based on linear response theory. It amounts to permanently adding an external forcing to the system, which induces a nonzero flux in the steady-state. The transport coefficient is then obtained by diving the average flux by the magnitude of the perturbation, for small values of the perturbation in order to ensure one stays in the linear response regime.
       
   There are several sources of error associated with this technique. In particular, the main concern is the statistical error, much larger than the usual asymptotic variance for standard time averages due to the small magnitude of the forcing. See \cref{rmk:NEMD_vs_TTCF} and~\cite[Section~5]{lelievre2016}, \cite[Section~2]{spacek2023} and \cite[Section~3]{leimkuhler2016} for a more detailed discussion on the numerical analysis of nonequilibrium methods.
   
    \item {\bf Transient methods.} While both Green--Kubo and nonequilibrium methods are based on steady-state dynamics, transient methods provide an alternative framework by monitoring the system’s relaxation to a steady-state after an initial perturbation, and can be classified into two main approaches.
    
    \begin{enumerate}
    \item {\bf Relaxation to equilibrium:} A typical scenario is to perturb an equilibrium system by creating an initial profile of momentum or energy, for instance, which is then allowed to relax to an equilibrium steady-state through the time-evolution of the equilibrium dynamics. Transport coefficients are obtained either by computing the integral of some response, or by monitoring transient profiles and matching them to the solution of a macroscopic effective PDE parametrized by the transport coefficient at hand by some form of inverse problem fitting. Examples include the method proposed in \cite{hulse2005} to compute the thermal conductivity, the momentum impulse relaxation method \cite{arya2000}, and the approach-to-equilibrium molecular dynamics method \cite{lampin2013}. See Section~\ref{subsec:transient} for a more precise presentation.
                
        \item {\bf Relaxation to the nonequilibrium steady-state:} In a somewhat dual approach, one can alternatively start with an equilibrium system and drive it towards a nonequilibrium steady-state by applying an external forcing to the dynamics. The relaxation to a nonequilibrium steady-state is then monitored, from which the transport coefficient can be obtained. This corresponds to the TTCF method~\cite{morriss1987,evans1988}, which generalizes the Green--Kubo relations to nonlinear regimes; see Section~\ref{subsec:TTCF} for a more thorough presentation.
    \end{enumerate}
\end{enumerate}
The limitations and drawbacks listed above suggest that there is space for alternative approaches, in particular in the context of variance reduction; this motivates the construction of the transient subtraction method in \cref{sec:transient}.

\subsection{Transient method: relaxation to equilibrium}
\label{subsec:transient}

As discussed in \cref{subsec:num_tech}, an alternative approach to the NEMD and GK methods for computing transport coefficients is based on transient dynamics. We describe here the first option, namely relaxation to the equilibrium state. The fundamental idea is that, instead of applying an external forcing to the dynamics, or computing correlations for the equilibrium dynamics, we start from an initially perturbed system, and monitor its relaxation to the reference state by evolving equilibrium dynamics.

\paragraph{Mathematical formulation} The transient method in this context relies on two main ingredients: (i) perturbing the distribution of initial conditions at order $\bigO(\eta)$ with $\eta\ll 1$; and (ii) monitoring return to stationarity via time integration. More precisely, we consider a process $X_t^\eta$ which evolves according to the reference dynamics \eqref{eq:general_SDE}, with $X_0^\eta \sim \psip$. The probability measure $\psip$ is assumed to be a first-order perturbation of the invariant probability measure of the reference dynamics $\mu$, satisfying
\begin{equation}
	\psip = (1+\eta S)\mu + \bigO(\eta^2).
	\label{eq:init_dist}
\end{equation}
We then evolve the process $X_t^\eta$ according to the reference equilibrium dynamics, which relaxes over time to its equilibrium steady-state. In particular, although not immediately clear, the time integral of the expectation of $R(X_t^\eta)$, when divided by $\eta$, converges to the transport coefficient $\rho$ as $\eta$ goes to~$0$:
\begin{equation}
	\rho = \lim_{\eta\to 0} \frac{1}{\eta}\int_0^{+\infty} \E(R(X_t^\eta)) \, dt.
	\label{eq:transient_def}
\end{equation}
To motivate the equality \eqref{eq:transient_def}, we consider finite $\eta\ll 1$. By writing the expectation in terms of the semigroup, and using that $(\e^{t\L}R)(x) = \E^x[R(X_t)]$ (the expectation being with respect to realizations of~\eqref{eq:general_SDE} started from~$X_0=x$) has average 0 with respect to $\mu$ (by the invariance of $\mu$ by the dynamics and the fact that $R$ has average 0 with respect to $\mu$), we have, informally,
\begin{align}
    \label{eq:gk_equiv1}
	\frac{1}{\eta}\int_0^{+\infty} \E(R(X_t^\eta)) \, dt &= \frac{1}{\eta} \int_0^{+\infty} \int_\mathcal{X} \bigl(\e^{t\L} R\bigr) \, d\psip \, dt \\
	\notag
	&= \frac{1}{\eta}\int_0^{+\infty} \int_\mathcal{X} \e^{t\L} R \, d\mu \, dt + \int_0^{+\infty} \int_\mathcal{X} \bigl(\e^{t\L} R\bigr) S \, d\mu \, dt + \bigO(\eta) \\
	\notag
	&= \int_0^{+\infty} \int_\mathcal{X} \bigl(\e^{t\L} R\bigr) S \, d\mu \, dt + \bigO(\eta) \\
	&= \int_0^{+\infty} \E_\mu\bigl(R(X_t)S(X_0)\bigr) dt + \bigO(\eta).
    \label{eq:gk_equiv2}
\end{align}
It is clear that by letting $\eta\to 0$ we get the correct result, i.e., \eqref{eq:transient_def} is equivalent to the Green--Kubo formula \eqref{eq:gk}:
\begin{equation}
    \notag
	\lim_{\eta\to 0}\frac{1}{\eta}\int_0^{+\infty} \E(R(X_t^\eta)) \, dt = \int_0^{+\infty} \E_\mu(R(X_t) S(X_0)) \, dt.
\end{equation}
We recall that $\E_\mu$ denotes the expectation with respect to the reference dynamics started at equilibrium, while $\E$ on the left-hand side denotes the expectation with respect to the reference dynamics initialized as $X_0^\eta \sim \psip$.

The above discussion is an informal presentation of the method, and is done for motivational purposes; see \cref{prop:gen_subtraction} for the formal meaning of the initial distribution \eqref{eq:init_dist}, and the rigorous form of the computation \eqref{eq:gk_equiv1}--\eqref{eq:gk_equiv2}.  

\paragraph{Estimators of transient dynamics relaxing to equilibrium} In practice, numerically estimating \eqref{eq:transient_def} requires first approximating the limit with (sufficiently small) finite $\eta$, truncating the time integral to finite $T$, and approximating the expectation with an average over $K$ realizations of the dynamics started from i.i.d.\ initial conditions $X_0\sim\psip$. This leads to the following estimator for \eqref{eq:transient_def}:
\begin{equation}
	\Test = \frac{1}{\eta K}\sum_{k=1}^K \int_0^T R(X_t^{\eta,k}) \, dt.	
	\label{eq:T_estimator}
\end{equation}
Although these approximations lead to several sources of bias in \eqref{eq:T_estimator}, which are made precise in \cref{subsubsec:bias_analysis}, the primary concern associated with \eqref{eq:T_estimator} is its very large variance, as we discuss next. This disqualifies it as an appropriate numerical method.

\paragraph{Asymptotic variance of usual transient estimator} The asymptotic variance of the estimator~\eqref{eq:T_estimator} is
\begin{equation}
    \notag
	\lim_{T\to+\infty} T^{-1}\Var\bigl(\Test\bigr) = \frac{2}{K\eta^2}\int_\mathcal{X} R\paren*{-\L^{-1} R} \, d\mu.
\end{equation}
It corresponds to the usual asymptotic variance for time averages of ergodic equilibrium dynamics, except for the very large prefactor $1/\eta^2$. 

Unlike the usual NEMD or Green--Kubo estimators of transport coefficients discussed in \cref{subsec:num_tech}, the variance of~\eqref{eq:T_estimator} is magnified by two distinct contributions. First, as with NEMD, we divide~\eqref{eq:T_estimator} by $\eta\ll 1$ which gives rise to the $\bigO(\eta^{-2})$ factor. Second, since the estimator is not a time average but a time integral as with GK and TTCF, the variance also scales linearly in~$T$, as opposed to the typical scaling~$\bigO(1/T)$ for time-averages. This leads to variance of order~$\bigO(T\eta^{-2})$, much higher than its NEMD and GK counterparts. 

This result calls for modifying the estimator with the use of variance reduction techniques, in particular to get rid of the $\eta^{-2}$ contribution and obtain bounds uniform in~$\eta$ as for TTCF. To this end, we consider the use of \emph{couplings} as a control variate, which are discussed more precisely in Section~\ref{sec:transient}.

\subsection{Transient method: relaxation to the nonequilibrium steady-state}
\label{subsec:TTCF}

Transient-time correlation functions~\cite{morriss1987,evans1988} can be seen as a dual approach to relaxation-to-equilibrium transient techniques: instead of starting from an initially perturbed state and relaxing towards equilibrium with the reference dynamics, the system is started at equilibrium and evolves according to the nonequilibrium dynamics. We provide in this section a formal derivation of the method, formulated for general dynamics, which can be deterministic or stochastic (as in~\cite{pincus2023} for instance).

\paragraph{Derivation of the TTCF method}

Consider the following nonequilibrium dynamics started under the reference distribution~$\mu$: 
\begin{equation}
  \label{eq:noneq_dynamics_started_eq}
  dY_t^\eta = \left( b(Y_t^\eta) + \eta F(Y_t^\eta) \right) dt + \sigma dW_t, \qquad Y_0^\eta \sim \mu,
\end{equation}
where~$F:\X\to\R^d$ is some external driving on the system (assumed to be a bulk driving for simplicity of exposition). We assume that~\eqref{eq:noneq_dynamics_started_eq} admits a unique invariant probability measure, denoted by~$\mu_\eta$ (with~$\mu_0 = \mu$). Denoting by~$\G_\eta = \L + \eta \wcL$ with~$\wcL = F^\top \nabla$ the generator of the above dynamics, the average value of some response function~$R$ at time~$T$ is
\[
\E\left[R(Y_T^\eta)\right] = \int_\X \e^{T \G_\eta} R \, d\mu = \int_\X R \, d\mu + \int_0^T \int_\X \G_\eta \e^{t \G_\eta}R \, d\mu \, dt,
\]
where we made use of the operator identity 
\[
\e^{T \G_\eta} = \Id + \int_0^T \G_\eta \e^{t \G_\eta} \, dt.
\]
Now, recalling that the conjugate response introduced in Section~\ref{subsec:gen_setting} is actually equal to (see~\cite[Section 5.2.3]{lelievre2016})
\begin{equation}
  \label{eq:conjugate_response}
  S = \wcL^* \mathbf{1},
\end{equation}
where adjoints are taken on~$L^2(\mu)$ (see \eqref{eq:Astar_adjoint} below for a more precise definition), we find that~$\G_\eta^* \mathbf{1} = \eta S$ (since~$\L^* \mathbf{1} = 0$ by the invariance of~$\mu$ under the reference dynamics), and therefore
\[
\E\left[R(Y_T^\eta)\right] = \int_\X R \, d\mu + \int_0^T \int_\X \left(\e^{t \G_\eta}R\right) \G_\eta^* \mathbf{1} \, d\mu \, dt = \int_\X R \, d\mu + \eta \int_0^T \E\left[ R(Y_t^\eta) S(Y_0^\eta)\right] \, dt.
\]
This equality can be rewritten as
\[
\frac{\E\left[R(Y_T^\eta)\right] - \mu(R)}{\eta} = \int_0^T \E\left[ R(Y_t^\eta) S(Y_0^\eta)\right] \, dt.
\]
By passing to the limit~$T \to +\infty$, denoting by~$\E_\eta$ the expectation with respect to the steady-state probability measure of the nonequilibrium dynamics~$(Y_t^\eta)_{t \geq 0}$, we can conclude that
\begin{equation}
  \label{eq:NL_response}
  \frac{\E_\eta(R)-\E_0(R)}{\eta} = \int_0^{+\infty} \E\left[ R(Y_t^\eta) S(Y_0^\eta)\right] \, dt,
\end{equation}
where the expectation on the right-hand side is with respect to the dynamics~\eqref{eq:noneq_dynamics_started_eq} (started under the equilibrium distribution~$\mu$ and relaxing towards the steady-state~$\mu_\eta$). The time integral represents the cumulated effect of a transient phase of the dynamics. Note that~\eqref{eq:NL_response} gives the exact (nonlinear) response, whatever the value of~$\eta \neq 0$.

\begin{remark}[Relationship with Green-Kubo formulas]
  The usual Green--Kubo formula is recovered in the limit~$\eta \to 0$ in~\eqref{eq:NL_response}. Note however that, compared to GK simulations, TTCF allows to recover response profiles as in NEMD, by computing for instance local thermodynamic properties based on~\eqref{eq:NL_response} (for instance, momentum profiles for shear flows, kinetic temperature profiles for thermal transport, etc).
\end{remark}

\paragraph{Estimators of the nonlinear response}

In view of~\eqref{eq:NL_response}, the natural TTCF estimator
\begin{equation}
  \label{eq:TTCF_estimator}
  \TestTTCF = \frac{1}{K} \sum_{k=1}^K \int_0^T R(Y_t^{\eta,k}) S(Y_0^{\eta,k}) \, dt, \qquad Y_0^{\eta,k} \sim \mu,
\end{equation}
with independent initial conditions, is asymptotically unbiased as~$t \to +\infty$ for any value of~$\eta$ (whatever large), and allows therefore to capture the full nonlinear response (in contrast to the transient methods based on relaxation to equilibrium where some care in the design of the initial condition is required in order to limit the~$\eta$ bias). This is admittedly a strength of TTCF approaches. Moreover, the statistical error of~\eqref{eq:TTCF_estimator} can be bounded uniformly in~$\eta$ for~$|\eta|$ small. It however diverges with the integration time, similary to Green-Kubo type methods, as discussed more precisely in Section~\ref{sec:error_TTCF}.

Two options are considered in the literature in order to perform some variance reduction for this method:
\begin{itemize}
\item A first option is to rely on antithetic variables, encoded through an involution~$\mathscr{M}$ acting on the phase space variables (for instance changing the sign of some momenta), and possibly changing the sign of the conjugate response~$S$. Nonequilibrium trajectories are then started from initial conditions~$Y_0^{\eta,n}$ for~$1 \leq n \leq N$, and also from the transformed initial conditions~$\mathscr{M}Y_0^{\eta,n} = \widehat{Y}_0^{\eta,n}$. This goes under the name of trajectory mappings, and is discussed in~\cite[Section~7.4]{evans2007} and subsequent works making use of this idea such as~\cite{delhommelle2005, maffioli2022, maffioli2024}.
\item Another option is to perform some recentering of the function~$R$ by subtracting off (an estimator of) its steady-state average, as discussed around~\cite[Eq.~(3)]{maffioli2024}. This allows to reduce the scaling of the asymptotic variance of the estimator from~$t^2$ to~$t$; see Section~\ref{sec:error_TTCF}.  
\end{itemize}

\section{Transient subtraction method}
\label{sec:transient}

We propose in this section a method called \emph{transient subtraction technique}, which employs a subtraction technique similar to the one suggested in \cite{ciccotti1975} to the transient dynamics method discussed in \cref{subsec:transient} as a means for variance reduction. We first outline in \cref{subsec:constructing_method} the construction of the method, then present the numerical analysis of its associated estimators in \cref{subsec:num_anal_ts}. We finally provide some elements of numerical analysis for the TTCF method in \cref{sec:error_TTCF} in order to compare more precisely the transient subtraction technique we propose to TTCF (see the discussion in \cref{sec:conclusion}).

\subsection{Constructing the method}
\label{subsec:constructing_method}
In the transient dynamics setting of \cref{subsec:transient}, one can consider the use of couplings as a control variate approach to construct an estimator with lower variance than~\eqref{eq:T_estimator}. To this end, we consider the coupling $(X_t^\eta,Y_t^0)$, where the processes $X_t^\eta$ and $Y_t^0$ are evolved according to the same underlying reference dynamics and have different initial conditions:
\begin{equation}
\begin{cases}
\begin{aligned}
	dY_t^0 &= b(Y_t^0) \, dt + \sigma \, dW_t, \qquad Y_0^0 \sim \mu, \\
	dX_t^\eta &= b(X_t^\eta) \, dt + \sigma \, d\widetilde{W}_t, \qquad X_0^\eta \sim \psip,
\end{aligned}
\end{cases}
\label{eq:coupled_dynamics}
\end{equation}
where $W_t$ and $\widetilde{W}_t$ are standard $m$-dimensional Brownian motions. The transport coefficient $\rho$ can then be computed as
\begin{equation}
\rho = \lim_{\eta\to 0} \frac{1}{\eta}\int_0^{+\infty} \E\paren*{R(X_t^\eta) - R(Y_t^0)}\,  dt.
	\label{eq:tc_subtraction}
\end{equation}
Note that $\int_0^{+\infty} R(Y_t^0) \, dt$ acts as a control variate since $\E(R(Y_t^0))=0$ for all $t\geq 0$. The expression~\eqref{eq:tc_subtraction} admits the following natural estimator, carried out with independent initial conditions for the couple $(X_t^{\eta,k},Y_t^{0,k})_{t\geq 0}$ for $1\leq k\leq K$ and independent realizations of the dynamics~\eqref{eq:coupled_dynamics}:
\begin{equation}
    \TSest = \frac{1}{\eta K} \sum_{k=1}^K \int_0^T \bkt[\Big]{R(X_t^{\eta,k}) - R(Y_t^{0,k})} \, dt.
    \label{eq:ts_estimator}
\end{equation}
A sufficient condition for \eqref{eq:ts_estimator} to have smaller variance than the standard estimator~\eqref{eq:T_estimator} is for the trajectories to start $\eta$ close, and to stay close for times of order $1/\lambda$, where $\lambda$ is the relaxation rate of the system to the stationary state (see \cref{as:decay_semigroup}). More precisely,

\begin{enumerate}
	\item \label{enum:init_dist} The initial distance $|X_0^\eta - Y_0^0|$ should be of order $\eta$;
	\item \label{enum:stay_close} The dynamics should remain $\eta$ close for finite times as the copies of the system evolve, i.e., $|X_t^\eta - Y_t^0|$ must be of order $\eta$ for $t\leq T$.
\end{enumerate}
Condition \ref{enum:init_dist} amounts to finding a coupling measure which is concentrated along the diagonal in the~$(x,y)$ space, so that initial conditions are $\eta$ close. We emphasize that, although $\psip$ is by construction a $\bigO(\eta)$ perturbation of~$\mu$, this is not enough to guarantee that the trajectories start $\eta$ close when the initial conditions are independent, thus we require a coupling on the initial conditions.

 We discuss in \cref{subsubsec:synchronous_coup} a natural way of coupling the dynamics \eqref{eq:coupled_dynamics}, and outline sufficient conditions for condition \ref{enum:stay_close} to hold. Then, we formally construct the coupling measure on the initial conditions and discuss its properties in \cref{subsubsec:properties_init_cond}.
 
\begin{remark}[{\bf Tangent dynamics}]
The expression \eqref{eq:tc_subtraction} of the transport coefficient can be formulated in terms of tangent dynamics \cite{assaraf2017}. Denote by $\mathcal{T}_t\in\R^d$ the tangent vector, where
\begin{equation}
\notag
	\mathcal{T}_t 
= \lim_{\eta\to 0} \frac{X_t^\eta - X_t^0}{\eta}.
\end{equation} 
This vector evolves according to a random ordinary differential equation, obtained by linearizing \eqref{eq:general_SDE}. Moreover,
\eqref{eq:tc_subtraction} can be written as
\begin{equation}
\notag
	\rho = \lim_{\eta\to 0} \frac{1}{\eta}\int_0^{+\infty} \E\paren*{R(X_t^\eta) - R(Y_t^0)} \, dt = \int_0^{+\infty} \E\paren*{\mathcal{T}_t\cdot \nabla R(X_t^0)} \, dt.
\end{equation}
\end{remark}

\subsubsection{Synchronous coupling}
\label{subsubsec:synchronous_coup}
A natural way to ensure that the dynamics remain close is via synchronous coupling, which amounts to using the same Brownian motion for both processes, i.e., setting $\widetilde{W} = W$. It is known that synchronous coupling performs well in the presence of global dissipativity. Without it, however, trajectories decouple and we cannot control the coupling distance for long times. For the transient subtraction method, we do not require long-time results, as the relaxation time of the estimator~\eqref{eq:ts_estimator} is typically of order $\bigO(1/\lambda)$, with $\lambda$ the exponential convergence rate from \cref{as:decay_semigroup} below. Thus, this suggests that synchronous coupling is an admissible control variate as long as the relaxation time is smaller than the decoupling time.

In order to more precisely state some results on the coupling distance, we give a sufficient condition for trajectories to decouple at most exponentially in time in the following assumption.

\begin{assumption}
\label{as:contractivity}
	There exists $B\in\R$ such that the drift $b\colon \mathcal{X} \to \R^d$ satisfies
\begin{equation}
		\forall (x,y) \in \mathcal{X}^2, \qquad \langle x-y, b(x)-b(y)\rangle \leq B|x-y|^2.
		\label{eq:no_dissip}
	\end{equation}
\end{assumption}

A sufficient condition for \eqref{eq:no_dissip} to be satisfied is when the drift $b$ is globally Lipschitz with constant~$\bLip$, in which case $B = \bLip$. In some fortunate cases where $B<0$, the drift is globally dissipative. In particular, global dissipativity ensures uniform exponential decay of the coupling distance $|X_t^\eta - Y_t^0|$.  We next state a standard result providing an upper bound for how fast trajectories decouple based on the estimate considered in \Cref{as:contractivity}.

\begin{lemma}
	\label{lemma:decoupling_times}
	Suppose that \Cref{as:contractivity} holds. Then, almost surely,
\begin{equation}
\notag
\forall t\geq 0, \qquad |X_t^\eta - Y_t^0| \leq \e^{tB}|X_0^\eta - Y_0^0|.
	\end{equation}
\end{lemma}

\begin{proof}In order to bound the distance between the trajectories at time $t$ in terms of the initial distance, we first write, by It\^o's formula,
\begin{align*}
		d\bigl(|X_t^\eta - Y_t^0|^2\bigr) &= 2\ip{X_t^\eta - Y_t^0, dX_t^\eta - dY^0_t} \\
		&= 2\langle X_t^\eta - Y_t^0, b(X_t^\eta) - b(Y_t^0)\rangle \, dt \\
		&\leq 2B|X_0^\eta - Y_0^0|^2 \, dt.
	\end{align*}
Gr\"onwall's lemma then gives the claimed bound. \end{proof}

\subsubsection{Properties of initial conditions}
\label{subsubsec:properties_init_cond}
We consider a coupling measure $\mucoup(dx \, dy)$ with marginals $\psip(dx)$ and $\mu(dy)$. In order to ensure that the initial conditions are $\eta$ close, the coupling measure must be concentrated along the diagonal (more precisely, within $\eta$ distance from the diagonal), as illustrated in \cref{fig:coup_meas_mock}. A natural way of achieving this is to formulate~$X_0^\eta$ as a deterministic map of $Y_0^0$, i.e., to look for $\Phi_\eta \colon \mathcal{X} \to \mathcal{X}$ such that~$X_0^\eta = \Phi_\eta(Y_0^0)$, with~$\Phi_\eta$ close to the identity function. The function~$\Phi_\eta$ should be chosen such that
\begin{equation}
	\psip = \Phi_\eta \# \mu = (1+\eta S)\mu + \bigO(\eta^2),
	\label{eq:deterministic_map}
\end{equation}
where $\#$ denotes the image measure of $\mu$ by $\Phi_\eta$: For any bounded measurable test function~$\varphi\colon \X \to \R$,
\begin{equation}
    \notag
	\int_\mathcal{X} \varphi \, d\psip = \int_\mathcal{X} \varphi \circ \Phi_\eta \, d\mu.
\end{equation}
We look for a map $\Phi_\eta$ of the form
\begin{equation}
	\Phi_\eta(x) = x + \eta\varphi_1(x),
	\label{eq:Phi_map}
\end{equation}
where $\varphi_1$ is determined by \eqref{eq:deterministic_map}. It is in fact given by a solution to the partial differential equation (PDE) \eqref{eq:varphi1_PDE} below, as made precise in \cref{prop:gen_subtraction}.
Note that \eqref{eq:Phi_map} can be formulated as a map higher than first-order in $\eta$; see \cref{subsubsec:bias_analysis} for a discussion of this point.

The transient subtraction technique then amounts to evolving synchronously coupled equilibrium dynamics starting from initial conditions which are deterministically related:
\begin{equation}
\begin{cases}
\begin{aligned}
	dY_t^0 &= b(Y_t^0) \, dt + \sigma \, dW_t, \qquad Y_0^0 \sim \mu, \\
	dX_t^\eta &= b(X_t^\eta) \, dt + \sigma \, dW_t, \qquad X_0^\eta = \Phi_\eta(Y_0^0).
\end{aligned}
\end{cases}
\label{eq:sub_dyn}
\end{equation}
We next perform error analysis on the transient subtraction technique estimator \eqref{eq:ts_estimator} for~$(X_t^\eta)_{t\geq 0}$ and~$(Y_t^0)_{t\geq 0}$ given by~\eqref{eq:sub_dyn}.
 
\begin{figure}\centering
	\includegraphics[width=0.75\textwidth]{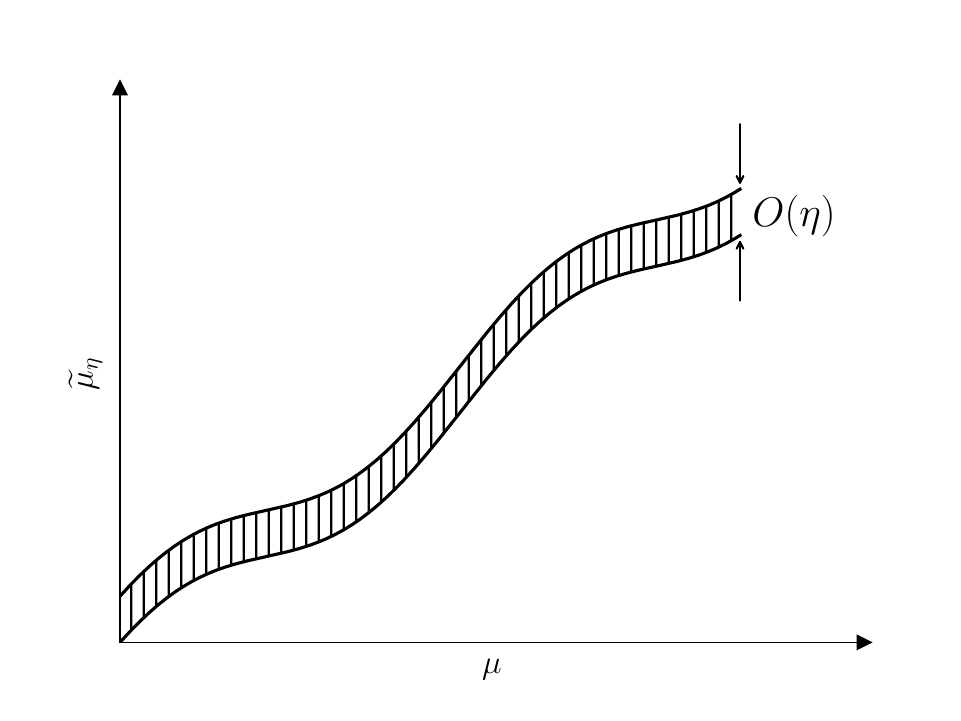}
	\caption{Illustration of coupling measure on initial conditions.}
	\label{fig:coup_meas_mock}
\end{figure}

\subsection{Numerical analysis of the transient subtraction method}
\label{subsec:num_anal_ts}
In this section, we perform error analysis on the transient subtraction estimator \eqref{eq:ts_estimator} realized with the dynamics~\eqref{eq:sub_dyn}. We start by making precise the functional setting and stating some estimates in \cref{subsubsec:lyapunov_setting}. We then make precise in \cref{subsubsec:bias_analysis} the bounds on the bias, and finally discuss its variance in \cref{subsubsec:variance_analysis}.

\subsubsection{Functional estimates}
\label{subsubsec:lyapunov_setting}
Consider a family of Lyapunov functions $(\mathcal{K}_n)_{n\in\N}$ with $\mathcal{K}_n \colon \mathcal{X} \to [1,+\infty)$ such that
\begin{equation}
\notag
	\forall n\in\N, \qquad \mathcal{K}_n \leq \mathcal{K}_{n+1}.
\end{equation}
The associated weighted $B^\infty$ spaces are
\begin{equation}
\notag
    B_n^\infty = \left\{\varphi \, \rm{measurable} \; \middle| \; \|\varphi\|_{B_n^\infty} := \sup _{x \in \mathcal{X}}\left| \frac{\varphi(x)}{\mathcal{K}_n(x)} \right| < +\infty\right\}.
\end{equation}
We next introduce the space $\S$ of smooth functions $\varphi$ belonging to the space $B^\infty_n$ for some $n$, and whose derivatives also belong to such spaces:
\begin{equation}
\notag
    \S = \left\{\varphi \in C^\infty(\mathcal{X}) \; \middle| \; \forall k \in \mathbb{N}^d, \; \exists n \in \mathbb{N}, \; \partial^k \varphi \in B_n^\infty\right\}.
\end{equation}
We finally define the subspace $\S_0$ of functions in $\S$ with average 0 with respect to $\mu$.

We make the following assumptions on the Lyapunov functions.

\begin{assumption}[{\bf Lyapunov estimates}]
\label{as:lyapunov}
There exist $n\in\N$ and~$C_n\in\R^+$ such that
\begin{equation}
	|x| \leq C_n\mathcal{K}_n(x).
	\label{eq:lyap_identity}
\end{equation}
Furthermore, for any~$n\in\N$,
\begin{equation}
	\norm{\mathcal{K}_n}_{L^1(\mu)} < +\infty.
	\label{eq:lyap_L1}
\end{equation}
Moreover, we assume that the Lyapunov functions are stable by products: for any~$n,n'\in\N$, there exist~$m\in\N$ and $C_{n,n'}\in\R^+$ such that
\begin{equation}
	\paren*{\mathcal{K}_n\mathcal{K}_{n'}}(x) \leq C'_{n,n'}\mathcal{K}_m(x).
	\label{eq:lyap_prod}
\end{equation}
We also assume stability by compositions: for any~$n,n'\in\N$ and~$\alpha^*\in\R^+$, there exist~$m\in\N$ and~$C_{n,n',\alpha^*}\in\R^+$ such that
\begin{equation}
\notag
	\forall\alpha \in [0,\alpha^*], \qquad \mathcal{K}_n(\alpha\mathcal{K}_{n'}(x)) \leq C_{n,n',\alpha^*}\mathcal{K}_m(x).
	\label{eq:lyap_comp}
\end{equation}
Lastly, we assume that $\mathcal{K}_n$ is nondecreasing in the following sense:
\begin{equation}
	\paren*{\forall i = 1,\dotsc,d, \quad |y_i| \leq |z_i|} \implies \mathcal{K}_n(y) \leq \mathcal{K}_n(z).
\end{equation}
\end{assumption}

A useful corollary of \cref{as:lyapunov}, which we will use in our estimates, is the following: for any~$f\in B^\infty_{n}$ and~$g = (g_1,\dotsc,g_d)$ with $g_i \in B^\infty_{n'}$, 
\begin{align}
\notag
	|f\circ g|(x) &\leq \|f\|_{B_{n}^\infty}\mathcal{K}_{n} \circ g(x) \\
	\notag
    &\leq \|f\|_{B_{n}^\infty}\mathcal{K}_{n}\paren*{\|g\|_{B^\infty_{n'}} \mathcal{K}_{n'}(x)} \\
    &\leq \norm{f}_{B^\infty_n}K_{n,n',\|g\|_{B^\infty_{n'}}}\mathcal{K}_{m}(x),
	\label{eq:lyap_composition_cor}
\end{align}
with $m$ depending on $n$ and $n'$.

A typical choice for $\mathcal{K}_n$ are polynomial Lyapunov functions of the form $\mathcal{K}_n(x) = 1+|x|^n$. This is a standard choice for Langevin dynamics; see \cite{mattingly2002,talay2002}. This choice satisfies \cref{as:lyapunov} when $\mu$ has moments of all orders, which is a mild requirement.

We also make an assumption on the convergence of the semigroup $\e^{t\L}$ in weighted~$B^\infty$ spaces. To this end, we introduce the subspace $B^\infty_n$ of functions with average~0 with respect to $\mu$:
\begin{equation}
\notag
    \Binfz = \left\{\varphi \in B^\infty_n \; \middle| \; \int_\mathcal{X} \varphi \, d\mu = 0\right\}.
\end{equation}

\begin{assumption}[{\bf Decay estimates on semigroup operator}]
  \label{as:decay_semigroup}
  For any $n\in\N$, there exist~$L_n \in \R^+$ and $\lambda_n>0$ such that
  \begin{equation}
    \forall \varphi\in \Binfz, \qquad \|\e^{t\L}\varphi\|_{B^\infty_n} \leq L_n\e^{-\lambda_n t}\norm{\varphi}_{B^\infty_n}.
    \label{eq:decay}
  \end{equation}
\end{assumption}

As a direct corollary of \cref{as:decay_semigroup}, the operator $\L$ is invertible on $\Binfz$, with
\begin{equation}
  \L^{-1} = -\int_0^{+\infty} \e^{t\L} \, dt.
  \label{eq:op_identity}
\end{equation}
Moreover, the following bound holds
\begin{equation}
\notag
  \norm{\L^{-1}}_{B^\infty_n} \leq \frac{L_n}{\lambda_n}.
\end{equation}
We refer for instance to \cite[Section 2]{lelievre2016} for a discussion on sufficient conditions for \cref{as:decay_semigroup} to hold (based on \cite{reybellet2006,hairer2011}), and for the proof of \eqref{eq:op_identity}.

\begin{remark}
Let us also emphasize that we consider systems for which the semigroup admits an exponentially decaying envelope, \emph{i.e.} for which the upper bound~\eqref{eq:decay} holds, with a decay rate possibly quite small. Correlation functions such as those appearing in the Green--Kubo formula~\eqref{eq:gk} or the TTCF formula~\eqref{eq:NL_response} are however typically unsigned for kinetic dynamics such as underdamped Langevin dynamics. Our estimate also allows for complex behaviors corresponding to the superposition of various exponential modes. It would be possible to extend our analysis to non-exponentially decaying correlations as long as they are integrable. We however refrain from doing so in order to keep the presentation simpler.   
\end{remark}

\subsubsection{Analysis of the bias}
\label{subsubsec:bias_analysis}
There are several sources of bias arising from the estimator~\eqref{eq:ts_estimator}, such as the time truncation and time discretization bias when considering numerical schemes to integrate the dynamics. Quantifying such biases is standard practice for estimators of this form. Additionally, there is a bias arising from the finiteness of $\eta$, which is the main result of this section. This is made precise in \cref{cor:gk_equiv}, which builds upon the estimates on the coupling measures provided by~\cref{prop:gen_subtraction} below. 

To state the result, we denote by $\mathcal{A}^*$ the adjoint of a closed operator $\mathcal{A}$ on $L^2(\mu)$: for any test functions~$\varphi, \phi \in C^\infty$ with compact support,
\begin{equation}
    \int_\mathcal{X} (\mathcal{A}\varphi)\phi \, d\mu = \int_\mathcal{X} \varphi(\mathcal{A}^*\phi) \, d\mu.
    \label{eq:Astar_adjoint}
\end{equation}

\begin{proposition}[{\bf Finite $\eta$ bias}]
\label{prop:gen_subtraction}
Suppose that \cref{as:lyapunov} holds true and that, for $S\in\S_0$, there exist solutions $\varphi_1 = (\varphi_{1,x_1},\dotsc,\varphi_{1,x_d}) \in (B^\infty_n)^d$ and $\varphi_2=(\varphi_{2,x_1},\dotsc,\varphi_{2,x_d}) \in (B^\infty_n)^d$ for some $n\in\N$ to the equations
\begin{equation}
	\nabla^*\varphi_1 = \sum_{i=1}^d \partial_{x_i}^* \varphi_{1,x_i} = S, \label{eq:varphi1_PDE}
\end{equation}
and
\begin{equation}
	\nabla^*\varphi_2 = -\frac{1}{2}\sum_{i,j=1}^d \partial_{x_i}^*\partial_{x_j}^* (\varphi_{1,x_i}\varphi_{1,x_j}) = -\frac{1}{2}(\nabla^*)^2 \colon \varphi_1\otimes \varphi_1.
	\label{eq:varphi2_PDE}
\end{equation}
Fix $\eta_*>0$, and $f\in \S$. Then, there exists $\mathcal{C}_{f,\eta_*} \in\R_+$ (which depends on~$f$ and~$\eta_*$) such that, for any~$|\eta| \leq \eta_*$,
\begin{equation}
	\qquad \left|\int_\mathcal{X} f\circ\Phi^\alpha_\eta \, d\mu - \int_\mathcal{X} f \, d\mu - \eta\int_\mathcal{X} f S \, d\mu\right| \leq \eta^{\alpha+1} \mathcal{C}_{f,\eta_*},
	\label{eq:prop_eta_bias_bound}
\end{equation}
with $\alpha=1$ or $\alpha=2$, and
\begin{equation}
\begin{cases}
\begin{aligned}
	&\Phi_\eta^1(x) = x + \eta\varphi_1(x), \\
	&\Phi_\eta^2(x) = x + \eta\varphi_1(x) + \eta^2\varphi_2(x).
\end{aligned}
\end{cases}
\label{eq:Phi_alpha}
\end{equation}
\end{proposition}

\begin{remark}
The smoothness condition on $f$ can be weakened, as it would be enough to have derivatives of $f$ up to order 3 in $B^\infty_m$. In order to simplify the presentation of the result, however, we suppose~$f\in \S$ here.
\end{remark}

This result states that the finite $\eta$ bias in the linear response is of order~$\bigO(\eta^{\alpha})$ for a map $\Phi_\eta$ which includes well-chosen terms up to order $\bigO(\eta^\alpha)$. It is of course possible to construct higher-order corrections in order to further decrease the bias. The associated PDEs for the corresponding $\varphi_i$ terms, however, become increasingly cumbersome to solve, rendering it an impractical approach. In any case, a second-order map leads to an estimator with $\bigO(\eta^2)$ bias, which is sufficiently small in general.

\begin{remark}[{\bf Well-posedness of PDEs}]
Although \eqref{eq:varphi1_PDE} might look difficult to solve, one can show that it admits gradient solutions of the form $\varphi_1 = \nabla\psi$ under some conditions on $\mu$. Indeed, \eqref{eq:varphi1_PDE} can then be written as
\begin{equation}
\notag
	\nabla^*\nabla \psi = S, 
	\end{equation}
which has a unique solution when $\nabla^*\nabla$ has a spectral gap when considered as an operator on~$L^2(\mu)$ (implied by $\mu$ satisfying a Poincar\'e inequality). Note that the solutions~$\varphi_1$ and~$\varphi_2$ are defined up to an element of the kernel of $\nabla^*$, i.e., if $\nabla\psi$ is a solution to~\eqref{eq:varphi1_PDE}, then~$\varphi_1 = \nabla\psi + g$ with $\nabla^*g = 0$ is also an admissible solution. \end{remark}

\begin{proof}[Proof of \cref{prop:gen_subtraction}]
It suffices to prove the result for $\alpha=2$, from which the result for $\alpha=1$ can be trivially deduced. By a Taylor expansion of $f(\Phi_\eta(x))$,
\begin{align*}
		f(\Phi_\eta(x)) = f(x) &+ \eta \nabla f(x)^\t(\varphi_1(x) + \eta\varphi_2(x)) \\
		&+ \frac{\eta^2}{2}(\varphi_1(x) + \eta\varphi_2(x))^\t \nabla^2 f(x)(\varphi_1(x) + \eta\varphi_2(x)) \\
		&+ \frac{\eta^3}{6} \nabla^3f\left(\Theta_\eta(x)\right)\cdot (\varphi_1(x) + \eta\varphi_2(x))^{\otimes 3},
\end{align*}
where $\nabla^3 f$ denotes the third-order derivative tensor (and $\nabla^2 f$ denotes the Hessian), and~$\Theta_\eta(x)$ interpolates between~$x$ and~$\Phi_\eta(x)$:
\begin{equation}
\notag
	\Theta_\eta(x) = (1-\theta_\eta(x))x + \theta_\eta(x)\Phi_\eta(x), \qquad \theta_\eta(x) \in [0,1].
\end{equation}
Integrating the Taylor expansion above yields
\begin{equation}
\label{eq:prop_rhs_terms}
\begin{split}
	\int_\mathcal{X} f\circ\Phi_\eta \, d\mu &= \int_\mathcal{X} f \, d\mu + \eta \int_\mathcal{X}\nabla f^\t \varphi_1 \, d\mu \\
 &\qquad + \eta^2\int_\mathcal{X} \paren*{\varphi_2^\t \nabla f+\frac{1}{2}\varphi_1^\t(\nabla^2 f)\varphi_1} \, d\mu + \eta^3\mathcal{R}_{3,\eta},
 \end{split}
\end{equation}
with $\mathcal{R}_{3,\eta}$ given by 
\begin{equation}
\mathcal{R}_{3,\eta} = \int_\mathcal{X} \bkt*{\frac{1}{6}\nabla^3f\left(\Theta_\eta\right)\cdot (\varphi_1 + \eta\varphi_2)^{\otimes 3} + \varphi_1^\t (\nabla^2 f)\varphi_2 + \frac{\eta}{2}\varphi_2^\t (\nabla^2 f)\varphi_2} \, d\mu.
	\label{eq:R3eta_bias}
\end{equation}
We next show that the remainder term $\mathcal{R}_{3,\eta}$ is uniformly bounded. We start with the first right-hand side integrand term in \eqref{eq:R3eta_bias}. Since $\varphi_1,\varphi_2$ have components in $B^\infty_n$ for some $n\in\N$, we deduce that so does $\Phi_\eta\in B^\infty_n$ (since the identity is in $B^\infty_n$ upon possibly increasing $n$ in view of \eqref{eq:lyap_identity}), and thus also~$\Theta_\eta$. Therefore, since $f\in\S$ and in view of \eqref{eq:lyap_comp} and \eqref{eq:lyap_composition_cor}, there exist $m,n'\in\N$ and $C\in\R^+$ such that, for all~$x\in\mathcal{X}$ and all $\eta\in [-\eta_*,\eta_*]$,
\begin{align}
\notag
	\abs*{\nabla^3f\paren*{\Theta_\eta(x)}} \leq C\smash{\sum_{i,j,k=1}^d}\norm*{\partial_{x_i,x_j,x_k}^3 f}_{B_m^\infty}\mathcal{K}_{n'}(x),
\end{align}
so that 
\begin{align}
\notag
	&\left|\int_\mathcal{X} \nabla^3f\paren*{\Theta_\eta(x)}\cdot (\varphi_1(x) + \eta\varphi_2(x))^{\otimes 3} \, d\mu\right| \\
	\label{eq:rem_bounded}
&\qquad \leq C\smash{\sum_{i,j,k=1}^d}\norm*{\partial_{x_i,x_j,x_k}^3 f}_{B_m^\infty}\paren*{\norm{\varphi_1}_{B_n^\infty} + \eta\norm{\varphi_2}_{B_n^\infty}}^3\int_\mathcal{X} \mathcal{K}_n^3\mathcal{K}_{n'} \, d\mu,
\end{align}
where \eqref{eq:lyap_comp} ensures uniformity in $\eta$ for the bounds above. By \cref{as:lyapunov}, there exist~$C\in \R^+$ and~$\ell\in\N$ such that $\mathcal{K}_n^3\mathcal{K}_m \leq C\mathcal{K}_\ell$. Since $\mathcal{K}_\ell \in L^1(\mu)$ by \eqref{eq:lyap_L1}, we conclude that~\eqref{eq:rem_bounded} is uniformly bounded for $|\eta|\leq \eta_*$. A similar computation shows that the remaining two integrand terms in \eqref{eq:R3eta_bias} are also uniformly bounded. 

To prove~\eqref{eq:prop_eta_bias_bound}, it remains to show that the first and second-order terms in~$\eta$ in~\eqref{eq:prop_eta_bias_bound} and~\eqref{eq:prop_rhs_terms} coincide, which is by design of~$\varphi_1$ and~$\varphi_2$. Indeed, by taking $L^2(\mu)$-adjoints and in view of condition~\eqref{eq:varphi1_PDE} on~$\varphi_1$,
\begin{equation}
\notag
	\int_\mathcal{X} \nabla f^\t \varphi_1 \, d\mu = \int_\mathcal{X} f\nabla^* \varphi_1 \, d\mu = \int_\mathcal{X} f S \, d\mu.
\end{equation}
Similarly, applying $L^2(\mu)$-adjoints to the second-order $\eta$ terms in \eqref{eq:prop_rhs_terms} yields
\begin{align}
\notag
	\int_\mathcal{X} \paren*{\varphi_2^\t \nabla f + \frac{1}{2}\varphi_1^\t (\nabla^2f) \varphi_1} \, d\mu = \int_\mathcal{X} \paren*{\nabla^*\varphi_2 + \frac{1}{2}(\nabla^*)^2 \colon \varphi_1\otimes \varphi_1}f \, d\mu,
\end{align}
which vanishes when $\varphi_2$ satisfies \eqref{eq:varphi2_PDE}. This allows us to conclude the proof.
\end{proof}

The above result allows us to quantify the bias, as made precise in the corollary below. Before stating it, we require an additional assumption.

\begin{assumption}
\label{as:L_stability}
The generator $\L$ is invertible on $\S_0$. In other words, for any $\phi\in\S_0$, there exists a unique solution $\Psi \in \S_0$ to the Poisson equation $-\L\Psi = \phi$.
\end{assumption}

\cref{as:L_stability} can be shown to hold for overdamped and underdamped Langevin dynamics under certain conditions on the potential $V$ \cite{talay2002,kopec2014,kopec2015}, and is a standard result in the literature.

Applying \cref{prop:gen_subtraction} together with the decay estimates from \cref{as:decay_semigroup}, as well as \cref{as:L_stability}, to the transient subtraction estimator \eqref{eq:ts_estimator} yields the following result on the bias of the estimator \eqref{eq:ts_estimator}.

\begin{corollary}
\label{cor:gk_equiv}
Under the assumptions of \cref{prop:gen_subtraction} as well as~\cref{as:decay_semigroup,as:L_stability}, there exists~$C\in\R^+$ such that, for any $T>0$ and $\eta \in [-\eta_*, \eta_*] \setminus \{0\}$,
\begin{equation}
\notag
	\left|\E\paren*{\aTSest} - \rho\right| \leq C\paren*{\eta^\alpha + \frac{\e^{-\lambda T}}{\eta}},
\end{equation}
where $\aTSest$ is defined as \eqref{eq:ts_estimator} with the dynamics \eqref{eq:sub_dyn} and $\Phi_\eta^\alpha$ given by \eqref{eq:Phi_alpha}.
\end{corollary}

This result, obtained as a direct consequence of \cref{prop:gen_subtraction}, shows that the bias of the transient (subtraction) technique has two distinct contributions: an exponentially decaying bias term arising from the time truncation, as in the Green--Kubo method (however magnified by a $\eta^{-1}$ factor); and a bias of order~$\eta^\alpha$ due to the finiteness of $\eta$, corresponding to deviations from the linear regime.

\begin{remark}
	Although the constant $C$ does not depend on $T$ (which suggests taking $T$ as large as possible to minimize the truncation bias contribution), the variance depends on $T$ (see \cref{prop:var_ts} below). This calls for equilibrating between the two in order to have the smallest overall error.
\end{remark}

\begin{proof}[Proof of \cref{cor:gk_equiv}]
	Fix $\eta \in [-\eta_*, \eta_*] \setminus \{0\}$. Since $R$ has average 0 with respect to $\mu$, it holds that
\begin{align}
\notag
\left|\E\paren*{\aTSest} - \rho\right| &= \frac{1}{\eta}\Biggl|\int_0^T \E\bigl(R(X^\eta_t)\bigr) \, dt - \eta\rho\Biggr| \\
\notag
&= \frac{1}{\eta}\Biggl|\int_0^{+\infty} \E\bigl(R(X^\eta_t)\bigr) \, dt - \int_T^{+\infty} \E\bigl(R(X^\eta_t)\bigr) \, dt - \eta\rho\Biggr| \\
&\leq \frac{1}{\eta}\abs*{\int_0^{+\infty} \E\bigl(R(X^\eta_t)\bigr) \, dt - \eta\rho} + \frac{1}{\eta}\abs*{\int_T^{+\infty} \E\bigl(R(X^\eta_t)\bigr) \, dt}.
\label{eq:bias_prop_eq1}
\end{align}
We first consider the second term in \eqref{eq:bias_prop_eq1}. By the semigroup definition of the expectation and the fact that~$\psip = \Phi_\eta\#\mu$, it holds that
\begin{align}
\notag
		\abs*{\int_T^{+\infty} \E\bigl(R(X^\eta_t)\bigr) \, dt} &= \abs*{\int_T^{+\infty} \int_\mathcal{X} \bigl(\e^{t\L} R\bigr) \circ \Phi_\eta \, d\mu \, dt}.
\end{align}
To bound the above quantity, we apply the semigroup decay estimate~\eqref{eq:decay} and use~\eqref{eq:lyap_composition_cor}: 
\begin{align*}
		\abs*{\int_T^{+\infty} \E\bigl(R(X^\eta_t)\bigr) \, dt} &\leq \int_T^{+\infty} \int_\mathcal{X} \abs*{\bigl(\e^{t\L} R\bigr) \circ \Phi_\eta} \, d\mu \, dt \\
		&\leq \int_T^{+\infty} \int_\mathcal{X} \norm*{\e^{t\L} R}_{B^\infty_n} \mathcal{K}_n \circ \Phi_\eta \, d\mu \, dt \\
		&\leq \int_T^{+\infty} \norm*{\e^{t\L} R}_{B^\infty_n} \int_\mathcal{X} \mathcal{K}_n\paren*{\norm{\Phi_\eta}_{B^\infty_{n'}} \mathcal{K}_{n'}} \, d\mu \, dt \\
		&\leq \norm{R}_{B^\infty_n}\int_T^{+\infty} L_n C_{n,n',\eta^*} \e^{-\lambda_n t} \paren*{\int_\mathcal{X} \mathcal{K}_m \, d\mu} \, dt \\
		&\leq \norm{R}_{B^\infty_n}\widetilde{C}_{m,n,n',\eta^*}\int_T^{+\infty} \e^{-\lambda_n t} \, dt \\
  &= \frac{\norm{R}_{B^\infty_n}\widetilde{C}_{m,n,n',\eta^*}\e^{-\lambda_n T}}{\lambda_n}. \end{align*}
We now consider the first term on the right-hand side of \eqref{eq:bias_prop_eq1}. 
Once again applying the semigroup definition of the expectation as well as the operator identity~\eqref{eq:op_identity}, it holds that
\begin{align}
		\int_0^{+\infty} \E\bigl(R(X^\eta_t)\bigr) \, dt = \int_\mathcal{X}\int_0^{+\infty} \bigl(\e^{t\L}R\bigr)\circ \Phi_\eta \, dt \, d\mu
= \int_\mathcal{X} \bigl(-\L^{-1}R\bigr)\circ \Phi_\eta \, d\mu.
		\label{eq:cor_Tint_to_LinvR}
	\end{align}
Writing the expectation in terms of the semigroup, and in view of the operator identity~\eqref{eq:op_identity}, we write the Green--Kubo formula~\eqref{eq:gk} as
\begin{align}
    \notag
		\rho = \int_0^{+\infty} \E_\mu\bigl(R(Y_t^0)S(Y_0^0)\bigr) \, dt = \int_\mathcal{X} (-\L^{-1}R)S \, d\mu.
	\end{align}
Thus, by \cref{prop:gen_subtraction} with $f = -\L^{-1}R \in \S_0$, it follows that
\begin{equation}
		\abs*{\frac{1}{\eta}\int_0^{+\infty} \E\bigl(R(X^\eta_t)\bigr) \, dt - \rho} \leq \eta^{\alpha} \mathcal{C}_{-\L^{-1}R,\eta_*}.
		\label{eq:cor_eta_bias}
	\end{equation}
This allows us to obtain the desired result.
\end{proof}

\subsubsection{Analysis of the variance}
\label{subsubsec:variance_analysis}
We state in this section some results on the scaling of the variance of the estimator~\eqref{eq:ts_estimator} used with the dynamics \eqref{eq:sub_dyn}.

\begin{proposition}[{\bf Variance of transient subtraction estimator}]
	\label{prop:var_ts}
	Suppose that \cref{as:contractivity} holds and that $R$ is globally Lipschitz with Lipschitz constant $\RLip$. Then, for any $T>0$,
\begin{equation}
		\Var\paren*{\aTSest} \leq \frac{\RLip^2}{K}\frac{\E\bkt*{|X_0^\eta - Y_0^0|^2}}{\eta^2}\paren*{\int_0^T \e^{tB} \, dt}^2.
\label{eq:var_prop_bound}
	\end{equation}
\end{proposition}
This result suggests that the variance grows at most exponentially fast in time, which is the case when $B>0$. For dissipative drifts, i.e., $B<0$, the variance is uniformly bounded in time by $B^{-2}$. Lastly, for $B=0$ the variance grows linearly in $T$.

Note that the variance is uniformly bounded in $\eta$. In particular, since the functions~$\varphi_1,\varphi_2$ (defined in~\eqref{eq:varphi1_PDE} and~\eqref{eq:varphi2_PDE}, respectively, and assumed to be in $B^\infty_n$ for some $n\in\N$) belong to~$L^2(\mu)$ by~\eqref{eq:lyap_L1}--\eqref{eq:lyap_prod}, it holds that
\begin{equation}
	\frac{\E\bkt*{\abs{X_0^\eta - Y_0^0}^2}}{\eta^2} = 
	\begin{cases}
		\norm{\varphi_1}^2_{L^2(\mu)} &\text{if }\alpha = 1, \\
		\norm{\varphi_1 + \eta\varphi_2}^2_{L^2(\mu)} &\text{if }\alpha = 2.
	\end{cases}
\label{eq:delta0_var_statement}
\end{equation}
Plugging \eqref{eq:delta0_var_statement} into \eqref{eq:var_prop_bound} immediately implies a bound on the variance uniform in the perturbation parameter $|\eta| \leq \eta_*$.

For the result stated in \cref{prop:var_ts} , we do not consider the asymptotic variance as for the transient estimator in \cref{subsec:transient}, as we cannot observe the $T\to+\infty$ limit due to the couplings we consider. As discussed in \cref{subsubsec:synchronous_coup}, the dynamics will decouple at large times, leading to a substantial increase in variance. We thus need to truncate the integration time $T$ and provide a nonasymptotic bound.

\begin{remark}
Variance reduction is obtainable for synchronous coupling even without strong conditions on the drift of~\eqref{eq:general_SDE}. In particular, when we have no dissipativity (i.e., $B>0$), the subtraction technique is better than the transient method discussed in \cref{subsec:transient} provided that~$\e^{2BT} \ll T/\eta^2$, i.e., $T\ll -\log(\eta)/B$. As one would typically consider $\eta\ll 1$, this suggests that the subtraction technique should therefore be preferred. 
\end{remark}

\begin{proof}[Proof of \cref{prop:var_ts}]
It suffices to consider the estimator \eqref{eq:ts_estimator} for $K=1$, since $\Var(\aTSest) = K^{-1}\Var(\widehat{\rho}^{T,1,\eta,\alpha}_{\rm sub})$. To simplify the notation, we write $\asTSest$ instead of $\widehat{\rho}^{T,1,\eta,\alpha}_{\rm sub}$:
\begin{equation}
\notag
		\asTSest = \frac{1}{\eta}\int_0^T (R(X_t^\eta) - R(Y_t^0)) \, dt.
	\end{equation}
Since $R$ is Lipschitz, and using \cref{lemma:decoupling_times} to bound the coupling distance $|X_t^\eta - Y_t^0|$ in terms of the initial distance $|X_0^\eta - Y_0^0|$, we have
\begin{align}
\notag
		\abs*{\asTSest} \leq \RLip \int_0^T \frac{|X_t^\eta - Y_t^0|}{\eta} \, dt \leq \RLip \int_0^T \frac{\e^{tB}|X_0^\eta - Y_0^0|}{\eta} \, dt.
\end{align}
We next bound the variance as
\begin{align}
\notag
		\Var\paren*{\asTSest} 
\leq \E\bkt*{\abs*{\asTSest}^2} = \RLip^2\E\bkt*{\abs*{\int_0^T \frac{\e^{tB}|X_0^\eta - Y_0^0|}{\eta} \, dt}^2},
	\end{align}
which leads to the desired result.
\end{proof}

\subsection{Elements on the error analysis for TTCF}
\label{sec:error_TTCF}

We make precise here the performance of the TTCF estimator~\eqref{eq:TTCF_estimator}, as discussed for instance around~\cite[Eqs.~(16) and~(17)]{maffioli2024}. We can in fact rely on the analysis performed in~\cite{pavliotis2024} and~\cite[Propositions~2.2 and~2.3]{pavliotis2024}, both to quantify the bias and the variance.

More precisely, when decay estimates on the semigroup similar to those of Assumption~\ref{as:decay_semigroup} hold for the semigroup~$\e^{t \G_\eta}$ uniformly over~$\eta$ in a compact set, the bias in the nonlinear response is exponentially small with respect to the time~$t$, in view of the equality
\begin{equation}
  \label{eq:bias_TTCF}
  \E(\TestTTCF) - \frac{\E_\eta(R)-\E_0(R)}{\eta} = \int_T^{+\infty} \left(\e^{t\G_\eta} R \right) S \, d\mu.
\end{equation}

For the variance, we note that the estimator is similar to a standard Green--Kubo estimator, except that the dynamics which is considered is the nonequilibrium dynamics with forcing magnitude~$\eta$, and the initial conditions are not distributed according to the stationary distribution. Consider the estimator obtained by recentering~$R$ to its steady-state value:
\[
\TestTTCFcenter = \frac1K \sum_{k=1}^K \int_0^t \left[R(Y_s^{\eta,k}) - \E_\eta(R)\right] S(Y_0^{\eta,k}) \, ds.
\]
In order to quantify the variance of this estimator, we introduce the Poisson equation~$-\G_\eta \mathcal{R}_\eta = R - \E_\eta(R)$ (which admits a unique solution when decay estimates estimates on the semigroup~$\e^{t\G_\eta}$ hold), and perform It\^o calculus on~$\mathcal{R}_\eta(Y_y^{\eta,k})$ to write
\[
\TestTTCFcenter = S(Y_0^{\eta,k})\left[\mathcal{R}_\eta(Y_0^{\eta,k}) - \mathcal{R}_\eta(Y_t^{\eta,k}) + \int_0^t \nabla \mathcal{R}_\eta(Y_s^{\eta,k})^\top \sigma \, dW_s \right].
\]
The dominant term on the right hand side of the previous equality is the time integral. By a straightforward adaptation of the proof of~\cite[Proposition~2.3]{pavliotis2024}, one can then show that
\begin{equation}
  \label{eq:asymptotic_variance_TTCF}
  \lim_{t \to +\infty} \frac{\Var(\TestTTCFcenter)}{t} = \|S\|^2_{L^2(\mu)} \left\|\sigma^\top \nabla\mathcal{R}_\eta\right\|_{L^2(\mu_\eta)}.
\end{equation}
This result shows that the variance of the recentred TTCF estimator is uniformly bounded for~$\eta$ in compact sets, and is of order~$t$. If the estimator is not recentered, then~$\TestTTCF = \TestTTCFcenter + t \E_\eta(R) S(Y_0^\eta)$ has a variance of order~$\mathrm{O}(t^2)$ due to the term which is linear in~$t$. This suggests to rely in practice on recentered estimators such as~$\TestTTCFcenter$, with an empirical estimation of the recentering (similarly to what is done for sensitivity estimators such as those considered in~\cite{plechac2021} for instance).

\begin{remark}
  \label{rmk:NEMD_vs_TTCF}
  The variance estimate~\eqref{eq:asymptotic_variance_TTCF} can be compared to the one obtained for the NEMD estimator of the transport coefficient
  \[
  \TestNEMD = \frac{1}{\eta t} \int_0^t R(\widehat{Y}_s^\eta) \, ds, \qquad \widehat{Y}_0^\eta \sim \mu_\eta,
  \]
  which is of order~$1/(\eta^2 t)$ (see for instance the discussion in~\cite[Section~3.1]{stoltz2024} for further precisions). In order to have a good NEMD estimator, one needs the square root of the variance to be smaller than~$\E(\TestNEMD) = \mathrm{O}(\eta)$. This is measured by the signal to noise ratio, which is the ratio of~$\E(\TestNEMD)$ and the square root of the variance, and is therefore of order~$\eta^2 \sqrt{t}$. The signal to noise ratio for the recentered TTCF estimator is in contrast of order~$1/\sqrt{t}$. In fact, in order to make comparisons at fixed computational cost, one should compare one long NEMD trajectory of length~$Kt$ and~$K$ independent realizations of length~$t$ for the recentered TTCF, so that the corresponding signal to noise ratios are respectively~$\eta^2 \sqrt{Kt}$ and~$1/\sqrt{Kt}$. The comparison between these two quantities allows to decide when NEMD is more efficient than TTCF and conversely. TTCF is better for~$\eta$ small, while NEMD methods are better for~$\eta$ large, the threshold value scaling as~$1/\sqrt{Kt}$.
\end{remark}

\section{Application to Langevin dynamics}
\label{sec:numerical_lang}
To illustrate the theoretical results obtained in \cref{sec:transient}, we apply the transient subtraction technique to compute the mobility and shear viscosity for a Lennard--Jones fluid, and to a low-dimensional example, with the Langevin dynamics \eqref{eq:lang_dynamics} serving as the underlying dynamics for all cases. We present our numerical results in three parts:
\begin{itemize}
	\item In \cref{subsec:num_lang}, we formulate the transient subtraction technique for Langevin dynamics by making precise $\varphi_1$ and $\varphi_2$ for the conjugate responses $S$ of interest.
\item In \cref{subsec:num_1D}, we numerically illustrate the finite $\eta$ bias results from \cref{cor:gk_equiv}, which apply~\eqref{eq:prop_eta_bias_bound} to the subtraction estimator \eqref{eq:ts_estimator}. In particular, we demonstrate the bias scaling for first and second-order maps $\Phi_\eta^\alpha$. This is done with the one-dimensional Langevin dynamics, which allows to directly compute \eqref{eq:prop_eta_bias_bound}, at the level of operators, by discretizing the associated PDE. This enables a clear and effective demonstration of the result, which would have otherwise been challenging to achieve with usual stochastic approaches. \item Finally, we compute in \cref{subsec:num_LJ} the mobility and shear viscosity for a Lennard--Jones fluid. This aim is to demonstrate the usefulness and viability of the method in more practical, high-dimensional molecular dynamics settings, particularly by highlighting its variance reduction capabilities.
\end{itemize}

\subsection{Transient methods for Langevin dynamics}
\label{subsec:num_lang}
Although our transient method only considers equilibrium dynamics, it encodes the relevant nonequilibrium information through the conjugate  response function $S$, which is the key quantity allowing to obtain the transport coefficient. This is expressed through the first-order perturbation PDE \eqref{eq:varphi1_PDE}, whose solution depends on $S$. To define the latter function, let $F(q) \in \R^{d}$ represent an external forcing, chosen appropriately based on the transport coefficient under consideration. Particular choices for $F(q)$ are made precise for each scenario we consider in~\cref{subsec:num_LJ}. For all such scenarios, the associated conjugate response function $S$ is given by
\begin{equation}
    S(q,p) = \beta F(q)^\t M^{-1} p.
    \label{eq:conj_res}
\end{equation}
We remark that the formal definition of $S$ is based on the associated nonequilibrium dynamics, and relies on linear response theory to be rigorously derived. In the interest of clarity, we do not provide such a rigorous discussion, and instead refer the reader to \cite[Section 5.2.3]{lelievre2016} for a comprehensive discussion.

Having identified the appropriate conjugate response function~$S$, one can now construct the map~$\Phi_\eta^\alpha$, for~$\alpha=1,2$, by solving the associated PDEs~\eqref{eq:varphi1_PDE} and~\eqref{eq:varphi2_PDE}. 

\paragraph{First-order map $\varphi_1$} For convenience, let us first recall the expression for \eqref{eq:varphi1_PDE}: 
\begin{equation}
\notag
	\nabla^*\varphi_1 = \sum_{i=1}^d \partial_{x_i}^* \varphi_{1,x_i} = S.
 \end{equation}
For Langevin dynamics, it is natural to consider the position and momentum components of $\varphi_1$ by writing $\varphi_1(q,p) = (\varphi_{1,q}(q,p), \varphi_{1,p}(q,p))$, so that we can write $\nabla^*\varphi_1 = \nabla^*_q\varphi_{1,q} + \nabla^*_p\varphi_{1,p}$. More precisely, the action of the adjoint operators are given by
\begin{equation}
	\partial^*_{q_i} = - \partial_{q_i} + \beta\partial_{q_i} V , \qquad \partial^*_{p_i} = - \partial_{p_i} + \beta (M^{-1}p)_i,
	\label{eq:nabla_star_expressions}
\end{equation}
which can be obtained via integration by parts as in \eqref{eq:Astar_adjoint}. In view of \eqref{eq:nabla_star_expressions} and \eqref{eq:conj_res}, we can write~\eqref{eq:varphi1_PDE} more explicitly for Langevin dynamics as
\begin{equation}
\notag
	-\Div_q(\varphi_{1,q}) - \Div_p(\varphi_{1,p}) + \beta \nabla V^\t \varphi_{1,q} + \beta p^\t M^{-1} \varphi_{1,p} = \beta F(q)^\t M^{-1} p.
\end{equation}
Therefore, a natural solution for \eqref{eq:varphi1_PDE} in any dimension is
\begin{equation}
	\varphi_1(q,p) = \begin{pmatrix}
 	\varphi_{1,q}(q,p) \\ \varphi_{1,p}(q,p)
 	\end{pmatrix} =
	\begin{pmatrix}
 	0 \\ F(q)
 	\end{pmatrix},
 	\label{eq:varphi1_sol}
\end{equation}
and the transformation $\Phi^1_\eta$ is then given by
\begin{equation}
	\Phi_\eta^1(q,p) = 
	\begin{pmatrix}
 	  q \\ p + \eta F(q)
 	\end{pmatrix}.
  \label{eq:Phi1_map}
\end{equation}
Thus, constructing the initial conditions for a first-order transient trajectory simply amounts to shifting the initial momentum $p_0^0$ of some associated stationary equilibrium process by $\eta F(q_0^0)$.

\paragraph{Second-order map $\varphi_2$} Constructing the second-order map amounts to finding~$\varphi_2$ by solving~\eqref{eq:varphi2_PDE}, which we recall for convenience:
\begin{equation}
\notag
	\nabla^*\varphi_2 = -\frac{1}{2}\sum_{i,j=1}^d \partial_{x_i}^*\partial_{x_j}^* (\varphi_{2,x_i}\varphi_{2,x_j}) = -\frac{1}{2}(\nabla^*)^2 \colon \varphi_1\otimes \varphi_1.
\end{equation}
Substituting the solution \eqref{eq:varphi1_sol} for $\varphi_1$ in \eqref{eq:varphi2_PDE} leads to
\begin{align*}
    \nabla^*\varphi_2 &= -\frac{1}{2}(\nabla^*)^2\colon 
    \begin{pmatrix}
        0 \\ F    
    \end{pmatrix} \otimes 
    \begin{pmatrix}
        0 \\ F    
    \end{pmatrix} \equiv -\frac{1}{2}(\nabla_p^*)^2 \colon F\otimes F.
\end{align*}
Thus, as in the first-order case, one can choose $\varphi_{2,q} = 0$ so that $\varphi_2 = (0, \varphi_{2,p}(q,p))$. Next, recalling that~$\partial^*_{p_i} = -\partial_{p_i} + \beta (M^{-1}p)_i$,
\begin{align*}
    -\frac{1}{2}(\nabla_p^*)^2 \colon F\otimes F &= -\frac{1}{2}\sum_{i,j=1}^d \partial_{p_i}^*\partial_{p_j}^* \paren*{F_i F_j} \\
    &= -\frac{\beta}{2}\sum_{i,j=1}^d \bkt*{-\partial_{p_i} + \beta (M^{-1}p)_i}(M^{-1}p)_j F_i F_j \\
    &= -\frac{\beta^2}{2}\sum_{i,j=1}^d (M^{-1}p)_i(M^{-1}p)_j F_i F_j + \frac{\beta}{2} \sum_{i,j=1}^d \partial_{p_i}(M^{-1}p)_j F_iF_j \\
    &= -\frac{\beta^2}{2}\sum_{i,j=1}^d (M^{-1}p)_i(M^{-1}p)_j F_i F_j +\frac{\beta}{2}\sum_{i,j=1}^d [M^{-1}]_{j,i} F_i F_j \\
    &= -\frac{1}{2}\paren*{\beta p^\t M^{-1}F}^2 + \frac{1}{2}\beta F^\t M^{-1}F.
\end{align*}
Thus, a possible solution for the second-order term $\varphi_2$ is
\begin{equation}
\notag
    \varphi_2(q,p) = 
    \begin{pmatrix}
        0 \\ -\dfrac{\beta F(q)^\t M^{-1} p}{2}F(q)
    \end{pmatrix} = 
    \begin{pmatrix}
        0 \\ -\dfrac{1}{2}S(q,p)F(q)
    \end{pmatrix}.
\end{equation}
This leads to the second-order transformation
\begin{equation}
	\Phi_\eta^2(q,p) = 
	\begin{pmatrix}
 	  q \\ p + \eta F(q) - \dfrac{\eta^2}{2}F(q)S(q,p)
 	\end{pmatrix}.
  \label{eq:Phi2_map}
\end{equation}

\subsection{One-dimensional Langevin dynamics}
\label{subsec:num_1D}
We next present some numerical results showcasing the scaling of the finite $\eta$ bias for the first and second-order transformations $\Phi_\eta^\alpha$ derived in \cref{subsec:constructing_method}. As stated in \cref{cor:gk_equiv}, in particular~\eqref{eq:cor_eta_bias}, an estimator of order~$\alpha$ has bias~$\bigO(\eta^\alpha)$: 
\begin{equation}
\label{eq:bias_1D_Langevin}
	\abs*{\frac{1}{\eta}\int_0^{+\infty} \E\bigl(R(X^\eta_t)\bigr) \, dt - \rho} \leq \mathcal{C}\eta^{\alpha}.
\end{equation}
Note that we did not truncate the time-integral in the estimator above as the finite-time integration bias vanishes as $T\to+\infty$, allowing us to solely quantify the $\eta$ bias. In view of~\eqref{eq:cor_Tint_to_LinvR}, and denoting by $\mathcal{R}$ the solution to the Poisson equation $-\L\mathcal{R} = R$,  we can rewrite~\eqref{eq:bias_1D_Langevin} as
\begin{equation}
\abs*{\frac{1}{\eta}\int_\mathcal{X} \mathcal{R} \circ \Phi_\eta^\alpha \, d\mu - \int_\mathcal{X} \mathcal{R} S \, d\mu} \leq \mathcal{C}\eta^{\alpha},
\label{eq:bias_1d-lang}
\end{equation}
where we used that~$\mathcal{R} = -\L^{-1}R$ has average~0 with respect to~$\mu$. We write the bias in the form~\eqref{eq:bias_1d-lang} since the low-dimensionality of the system in consideration allows us to directly compute the bias by discretizing $\L$ and solving the PDE $-\L\mathcal{R} = R$. Note that the bias result presented here holds for both the naive transient \eqref{eq:T_estimator} and subtraction~\eqref{eq:ts_estimator} estimators.

\paragraph{Choice of observable} We consider $\mathcal{X} = 2\pi \T \times \R$ and the following observable, which has average 0 with respect to~$\mu$ by construction: 
\begin{equation}
\notag
	R(q,p) = \paren*{\cos(q) - \sin(q)} \, \e^{\beta V(q)}.
\end{equation}
This choice is also considered in order to avoid symmetries in the response function (which may occur for typical observables such as $p$ and $\nabla V$) so that the results are clearly presented; see \cite[Section~4.2]{spacek2023}, for a more detailed discussion regarding the symmetries and the observable. Furthermore, the forcing $F$ in consideration is a normalized constant force, i.e., $F = 1$.

\paragraph{Numerically estimating the bias} The low dimensionality of this example allows us to analytically compute \eqref{eq:bias_1d-lang} through a direct approximation of $\mathcal{R}$ via finite-difference methods, and the use of quadratures for the associated integrals over the phase-space; see \cite[Appendix B]{spacek2023} for precise details on the numerical implementation of the finite-difference scheme. The unbounded momentum domain is truncated to $[-p_\mathrm{max},p_\mathrm{max}]$, with $p_\mathrm{max} = 5$. The domain $[-\pi,\pi] \times [-p_\mathrm{max},p_\mathrm{max}]$ is then discretized into $m_q = 200$ by $m_p = 400$ points with uniform step sizes $\Delta q = 2\pi/m_q$ and $\Delta p = 2p_\mathrm{max}/(m_p-1)$, as we consider periodic boundary conditions in~$q$.

We consider two maps $\Phi_\eta^\alpha$, for $\alpha=1,2$, constructed from $\varphi_1$ and $\varphi_2$ obtained in \cref{subsec:num_lang}, as given by~\eqref{eq:Phi1_map} and~\eqref{eq:Phi2_map},
\begin{equation}
\notag
\begin{cases}
\begin{aligned}
	&\Phi_\eta^1(q,p) = (q, p + \eta), \\
	&\Phi_\eta^2(q,p) = \paren[\bigg]{q, p + \eta - \eta^2\frac{\beta}{2}p}.
\end{aligned}
\end{cases}
\end{equation}
The bias \eqref{eq:bias_1d-lang} was computed for various values of $\eta$. The results are shown in \cref{fig:1d_lang_bias} in a log-log scale, with reference lines included. This confirms that the bias associated with an $\alpha$-ordered map is itself of order $\alpha$, which is the main estimate of \cref{prop:gen_subtraction}.

\begin{figure}[ht]
	\centering
	\includegraphics[width=0.75\textwidth]{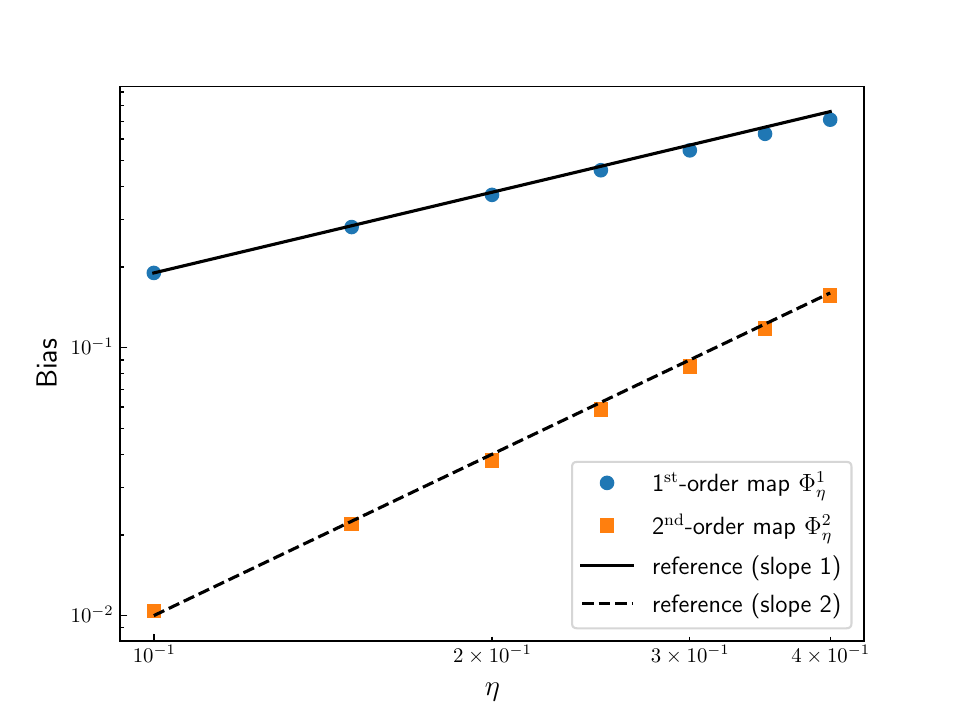}
	\caption{Bias \eqref{eq:bias_1d-lang} as a function of $\eta$ for the first and second-order maps, with overlayed reference lines.}
	\label{fig:1d_lang_bias}
\end{figure}

\subsection{Mobility and shear viscosity for Lennard--Jones fluids}
\label{subsec:num_LJ}
We next present some numerical results highlighting the variance-reduction potential of the transient subtraction method. The example in consideration is the computation of shear viscosity and mobility for a Lennard--Jones fluid. The system is composed of~$N$ particles in spatial dimension $D=3$ (so that~$d=3N$), evolving according to the underdamped Langevin dynamics \eqref{eq:lang_dynamics} on the domain $\mathcal{X} = (L\T)^{3N} \times \R^{3N}$. The potential energy corresponds to the sum of pairwise interactions
\begin{equation}
\notag
	V(q) = \sum_{1\leq i<j\leq N} v(\norm{q_i - q_j}),
\end{equation}
with $v(r)$ given by the standard 12-6 Lennard--Jones interaction potential:
\begin{equation}
	v(r) = 4\varepsilon\bkt*{\paren*{\frac{\sigma}{r}}^{12} - \paren*{\frac{\sigma}{r}}^6}.
	\label{eq:LJ_pairwise}
\end{equation}
The parameter $\varepsilon$ represents the depth of the potential well, and $\sigma$ determines the length scale (more precisely, $v'(2^{1/6}\sigma)=0$ so that interactions are attractive or repulsive depending on whether the distance is larger or smaller than $2^{1/6}\sigma$).
In practice, one truncates the range of~\eqref{eq:LJ_pairwise} at some value $r_\mathrm{c}$, after which interactions can be deemed negligible. We employ the truncated shifted-force cutoff method with a cutoff value of~$r_\mathrm{c}=2.5\sigma$, resulting in the modified potential
\begin{equation}
\notag
	v_\mathrm{SF}(r) = [v(r) - v(r_\mathrm{c}) - (r - r_\mathrm{c})v'(r_\mathrm{c})]\ind_{r\leq r_\mathrm{c}}.
\end{equation}
We numerically integrate the Langevin dynamics \eqref{eq:lang_dynamics} using the BAOAB splitting scheme~\cite{leimkuhler2013}, which allows to consistently sample the canonical measure. The simulations were conducted using with the \texttt{Molly.jl} package \cite{greener2024} in the Julia language, and were performed in dimensionless reduced units with $\sigma = \varepsilon = k_\mathrm{B} = 1$ and the mass matrix~$M = \mathrm{Id}$ on a cubic domain with side length $L = (N/\varrho)^{1/3}$, with $\varrho$ the particle density. For both shear viscosity and mobility computations, the results we present correspond to averages over $K = 10^5$ realizations of the system with i.i.d.\ initial conditions. 

We next describe the strategy for initializing and evaluating the trajectories, which is procedurally identical for the mobility and shear cases. Each independent realization of the system is initialized as follows. For the equilibrium control system, initial momenta were sampled from the Boltzmann--Gibbs measure, while initial positions were initialized on a cubic lattice. The system was then evolved for a thermalization time of $T_\mathrm{therm} = 1$ with a timestep size $\Delta t=10^{-3}$ in reduced units (the reference time being~$\sigma\sqrt{m/\varepsilon}$). We ensured that the thermalization time was sufficient long to melt the crystal structure and relax the system to a stationary state, as monitored by the stabilization of kinetic and potential energies, and by visual inspection of the molecular structure. Next, we initialize the transient trajectory by applying the transformation \eqref{eq:LJ_gen_map} to a copy of the stationary equilibrium system:
\begin{align}
    \begin{pmatrix}
        q_0^\eta \\ p_0^\eta
    \end{pmatrix}
	 = \Phi_\eta(q_0^0,p_0^0) = 
	\begin{pmatrix}
 	  q_0^0 \\ p_0^0 + \eta F(q_0^0)
 	\end{pmatrix},
 	\label{eq:LJ_gen_map}
\end{align}
where the expression for $F(q)$ is made precise for mobility and shear viscosity in \cref{subsubsec:mobility_num,subsubsec:shear}, respectively. The equilibrium and transient trajectories are then evolved simultaneously according to synchronously coupled standard equilibrium dynamics. The integration time $T$ should not be much larger than the relaxation time of the transient trajectory, as decoupling becomes a significant source of error. Nevertheless, this can be overcome during postprocessing, during which one can choose the appropriate truncation time for the estimator. 
Observational runs should be performed beforehand to have an approximate idea of the order of magnitude of the relaxation time (which varies significantly depending on the system at hand), which can be deduced from reasonably coarse and inexpensive runs. All simulation parameters are made precise in \cref{table:sim_params}.
\begin{table}[h!]
\centering
\begin{tabular}{@{}lcc@{}}
	\toprule
	\bf{Parameter} & {\bf Shear} & {\bf Mobility} \\
	\midrule
Integration time ($T$) & 3.5 & 2.0 \\
    Thermalization time ($T_\mathrm{therm}$) & 1 & 1 \\
	No. of realizations ($K$) & $10^5$ & $10^5$ \\
	Timestep ($\Delta t$) & $10^{-3}$ & $10^{-3}$ \\
	Inverse temp. ($\beta$) & 1.25 & 0.8 \\
	Damping ($\gamma$) & 1 & 1 \\
	No. of particles ($N$) & 1000 & 1000 \\
	Particle density ($\varrho$ [$N/L^3$]) & 0.7 & 0.6 \\
	LJ cutoff ($r_\mathrm{c}$) & 2.5 & 2.5 \\
	Mass matrix $M$ & {\rm Id} & {\rm Id} \\
	LJ param. $(\sigma)$ & 1 & 1 \\
	LJ param. $(\varepsilon)$ & 1 & 1 \\
	Boltzmann const. ($k_B$) & 1 & 1 \\
	\bottomrule
\end{tabular}
\caption{Simulation parameters}\label{table:sim_params}%
\end{table}

\subsubsection{Mobility}
\label{subsubsec:mobility_num}
When under the effect of a constant external field $F$, the mobility quantifies the particles' average velocity in the direction of the applied field. In our Lennard--Jones fluid example, we consider a constant force applied in the $x$-direction, and in particular we consider colored drift, which amounts to perturbing half the particles to one direction, and the other half in the opposite direction \cite{evans2007}:
\begin{equation}
\notag
	F = \frac{1}{\sqrt{N}}(F_1,F_2,\dotsc,F_N)^\t \in \R^{3N}, \qquad F_i = ((-1)^{i+1}, 0, 0), \qquad i = 1,\dotsc,N.
\end{equation}
The observable we consider is the velocity in the direction $F$, a standard choice for mobility computations \cite[Section 5.2.2]{lelievre2016}:
\begin{equation}
\notag
	R(q,p) = F^\t M^{-1}p.
\end{equation}
\begin{figure}[h!]
\centering
\begin{subfigure}{0.49\textwidth}
    \includegraphics[width=\textwidth]{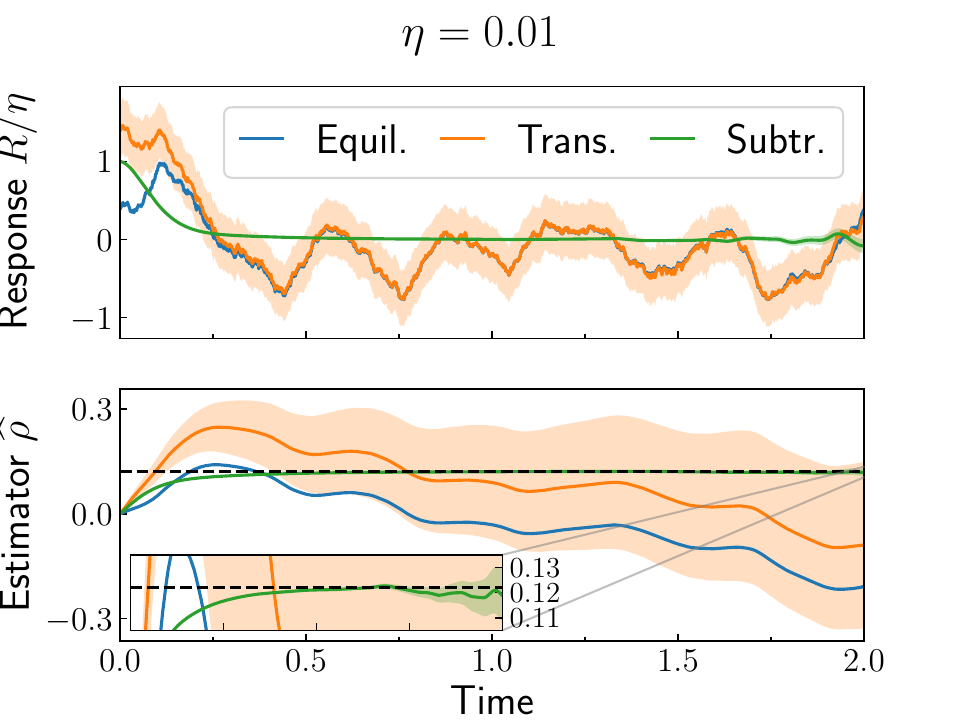}
    \caption{Data for $\eta = 0.01$.}
    \label{subfig:mobility_0.01}
\end{subfigure}
\hfill
\begin{subfigure}{0.49\textwidth}
    \includegraphics[width=\textwidth]{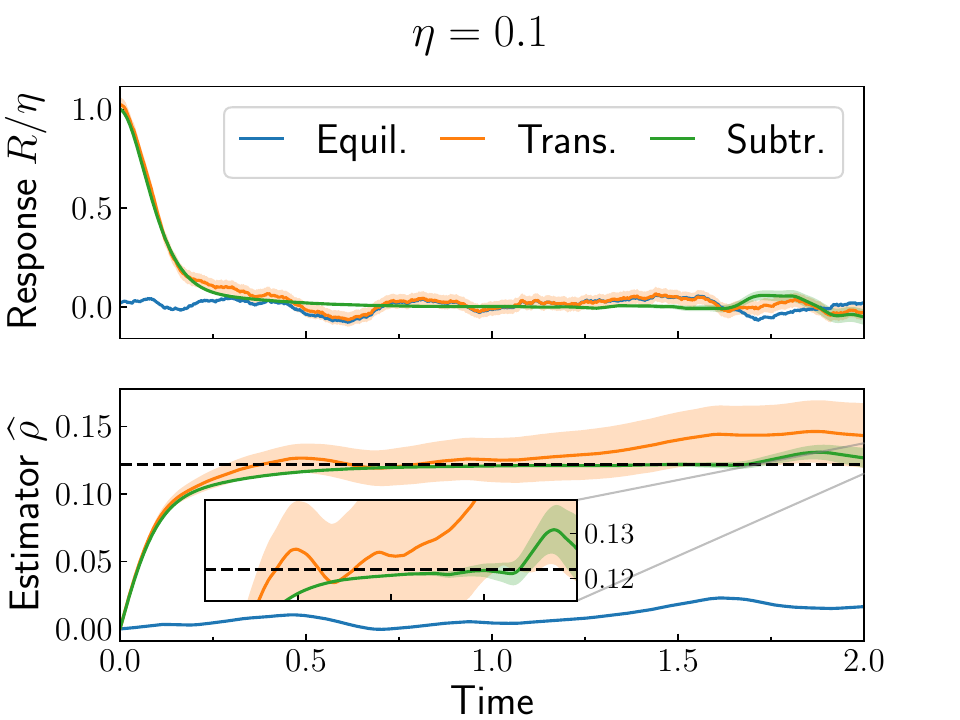}
    \caption{Data for $\eta = 0.1$.}
    \label{subfig:mobility_0.1}
\end{subfigure}
\caption{Trajectories for the computation of the mobility of a Lennard--Jones fluid with colored drift and associated error bars. The top graphs correspond to the instantaneous response (normalized by~$\eta$) as a function of time as transient trajectory relaxes, while the bottom graphs show the integrated response over time. The dashed line corresponds to the reference value~$\rho = 0.122$ obtained in~\cite{meier2004, blassel2024}.}
\label{fig:LJ_mobility}
\end{figure}

\begin{table}[h!]
\begin{subtable}[h]{0.45\textwidth}
\centering
\begin{tabular}{@{}lccc@{}}
\toprule
 & \multicolumn{2}{c}{\textbf{Variance at $T=1$}} &  \\
\cmidrule(lr){2-3}
 $\eta$ & \textbf{Naive} & \textbf{Subtraction} & \textbf{Ratio} \\
\midrule
0.01 & $\sn{2.66}{3}$ & $\sn{4.69}{-3}$ & $\sn{5.67}{5}$ \\
0.1  & 26.6   & $\sn{4.65}{-3}$ & $\sn{5.71}{3}$  \\
1.0  & 0.265    & $\sn{4.37}{-3}$ & 60.6    \\
\bottomrule
\end{tabular}
\caption{Data at $T=1$ (start of decoupling)}
\end{subtable}\hfill
\begin{subtable}[h]{0.45\textwidth}
\centering
\begin{tabular}{@{}lccc@{}}
\toprule
 & \multicolumn{2}{c}{\textbf{Variance at $T=2$}} &  \\
\cmidrule(lr){2-3}
 $\eta$ & \textbf{Naive} & \textbf{Subtraction} & \textbf{Ratio} \\
\midrule
0.01 & $\sn{5.70}{3}$ & 11.8 & 485 \\
0.1  & 57.2   & 5.52  & 10.4  \\
1.0  & 0.564    & 0.286  & 1.97   \\
\bottomrule
\end{tabular}
\caption{Data at $T=2$ (total decoupling)}
\end{subtable}
\caption{Comparison of variances between naive and subtraction transient estimators for various values of $\eta$ for the computation of mobility.}
\label{table:LJ_mobility}
\end{table}

Each trajectory is integrated for a physical time $T=2$. Although the system relaxes significantly before, we deliberately wanted to observe the decoupling point, which can be easily spotted in \cref{fig:LJ_mobility}, which presents the average trajectories for two values of $\eta$. The associated error bars shown in \cref{fig:LJ_mobility} were computed with empirical averages over the independent realizations. Due to the large signal-to-noise ratio of the mobility response, the subtraction's uniform bound in $\eta$ of the variance indeed shows to make a difference, as can readily be seen from \cref{fig:LJ_mobility}.

To quantitatively assess the variance reduction and the decoupling effect, we consider the variance values for both the naive transient and subtraction methods at two different times $T$: one right before trajectories decouple $T=1$, and one at the final time $T=2$. These results are summarized in \cref{table:LJ_mobility}. At $T=1$, we indeed see the variance's uniform bound in $\eta$ for the subtraction trajectory, while the $\eta^{-2}$ factor shows for the naive trajectory. Additionally, we notice that at $T=2$, even after significant decoupling, the subtraction method still provides significant variance reduction, even long after relaxation has occured.

\subsubsection{Shear viscosity}
\label{subsubsec:shear}
The shear viscosity of a fluid can be computed in a variety of ways; see \cite{todd2007} for a review on computational techniques. In this work, we consider a setting based on the transverse force-field method \cite{gosling1973,joubaud2012} with a sinusoidal forcing profile with spatial domain $(L_x\T \times L_y\T \times L_z\T)^N$. We denote by~$F_i\in\R^3$ the force acting on the $i$th particle:
\begin{equation}
\notag
    F_i = (f(q_{i,y}), 0, 0)^\t, \qquad  f(y) = \sin\paren*{\frac{2\pi y}{L_y}}.
\end{equation}
The force acts on the $x$-component of the momenta based on the particle's $y$-coordinate position.
The observable $R$ of interest is the imaginary part of the first empirical Fourier coefficient $U_1$:
\begin{equation}
    R(q,p) = \Im(U_1), \qquad U_1 = \frac{1}{N}\sum_{n=1}^N (M^{-1}p)_{n,x}\exp\paren*{\frac{2\mathrm{i}\pi q_{n,y}}{L_y}}.
    \label{eq:shear_observable}
\end{equation}
The initialization and evaluation of trajectories for this system were performed as described in \cref{subsec:num_LJ} with $L_x=L_y=L_z=L$. The numerical results are similar to those shown in \cref{subsubsec:mobility_num}, with largely the same interpretations and conclusion. A first difference, clearly seen from \cref{fig:LJ_shear}, however, is the magnitude of the error for the naive trajectories for which confidence intervals are much smaller than for \cref{fig:LJ_mobility}. This is a trivial artifact of the observable $R(q,p)$: for the mobility, the observable is $\bigO(\sqrt{N})$, while it is $\bigO(1)$ for the shear case, since \eqref{eq:shear_observable} corresponds to some spatial averaging. Secondly, the relaxation time for the shear trajectories is significantly longer compared to the one for mobility, and in fact almost coincides with the decoupling time. Nonetheless, \cref{table:LJ_shear} shows that variance reduction is still obtained.

\begin{figure}[h!]
\centering
\begin{subfigure}{0.49\textwidth}
    \includegraphics[width=\textwidth]{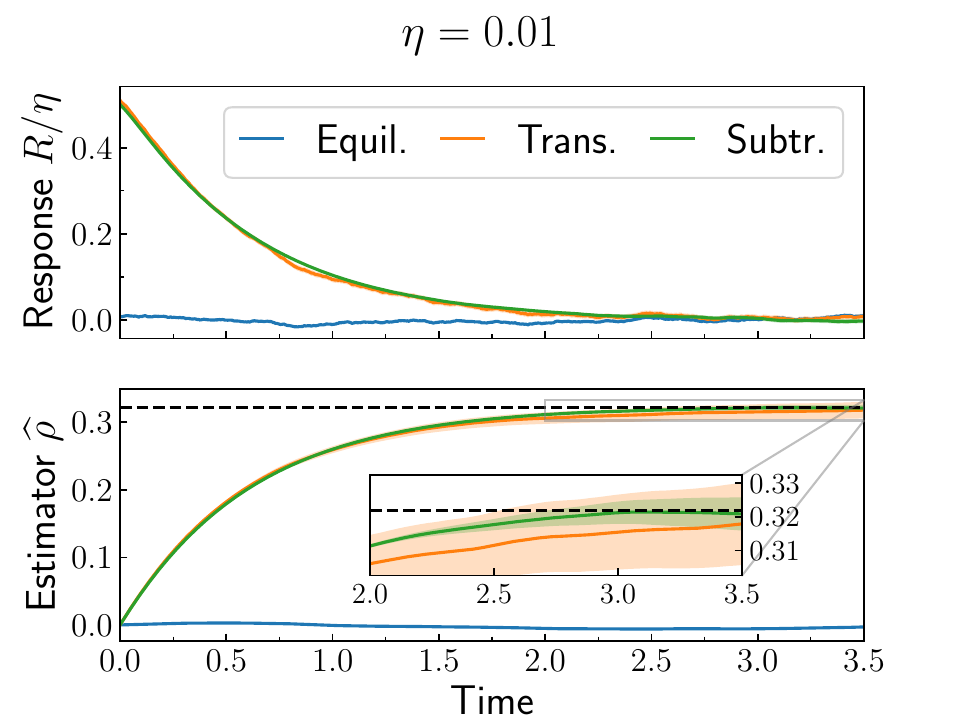}
    \caption{Data for $\eta = 0.01$.}
    \label{subfig:shear_0.01}
\end{subfigure}
\hfill
\begin{subfigure}{0.49\textwidth}
    \includegraphics[width=\textwidth]{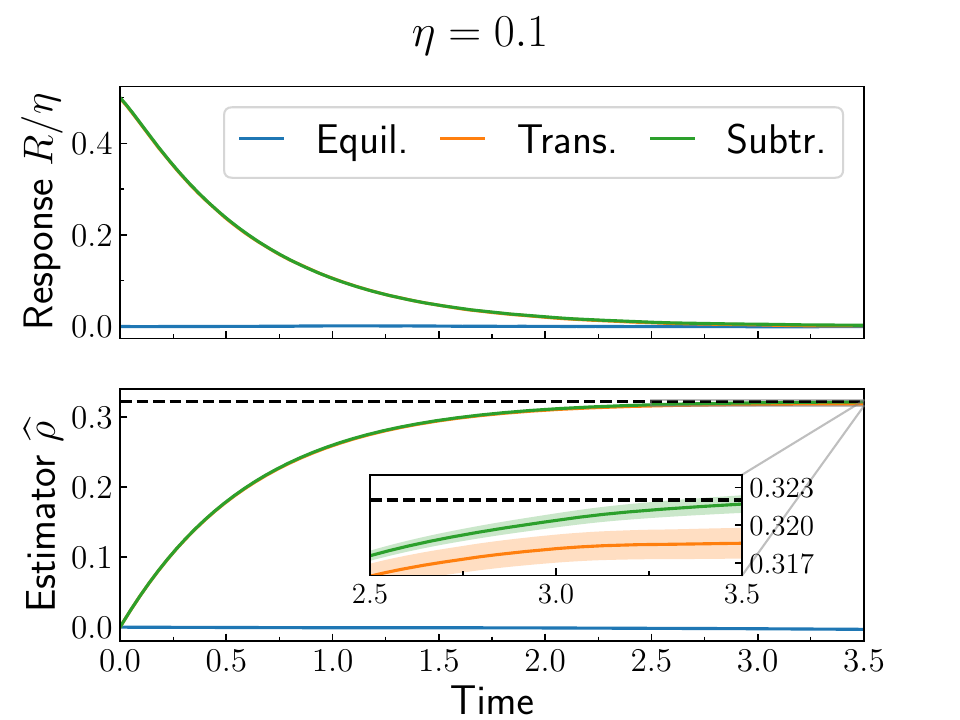}
    \caption{Data for $\eta = 0.1$.}
    \label{subfig:shear_0.1}
\end{subfigure}
\caption{Trajectories for the computation of the shear viscosity of a Lennard--Jones fluid and associated error bars. The top graphs correspond to the instantaneous response (normalized by~$\eta$) as a function of time as transient trajectory relaxes, while the bottom graphs show the integrated response over time. The dashed line corresponds to the reference value $U_1 = 0.322$ found in~\cite{blassel2024}.}
\label{fig:LJ_shear}
\end{figure}

\begin{table}[h!]
\begin{subtable}[h]{0.45\textwidth}
\centering
\begin{tabular}{@{}lccc@{}}
\toprule
 & \multicolumn{2}{c}{\textbf{Variance at $T=2$}} &  \\
\cmidrule(lr){2-3}
$\eta$ & \textbf{Naive } & \textbf{Subtraction } & \textbf{Ratio} \\
\midrule
0.01 & 7.32 & 0.0260 & 281 \\
0.1 & 0.0726 & $\sn{5.49}{-3}$ & 13.2 \\
\bottomrule
\end{tabular}
\caption{Data at $T=2$ (start of decoupling)}
\end{subtable}\hfill
\begin{subtable}[h]{0.45\textwidth}
\centering
\begin{tabular}{@{}lccc@{}}
\toprule
 & \multicolumn{2}{c}{\textbf{Variance at $T=3.5$}} &  \\
\cmidrule(lr){2-3}
 $\eta$ & \textbf{Naive } & \textbf{Subtraction } & \textbf{Ratio} \\
\midrule
0.01 & 15.0 & 2.45 & 6.12 \\
0.1  & 0.150  & 0.0487 & 3.07 \\
\bottomrule
\end{tabular}
\caption{Data at $T=3.5$ (total decoupling)}
\end{subtable}
\caption{Comparison of variances between naive and subtraction transient estimators for various values of $\eta$ for the computation of shear viscosity.}
\label{table:LJ_shear}
\end{table}

\section{Conclusion and perspectives}
\label{sec:conclusion}

We presented a variance reduction method to compute transport coefficients based on a transient approach with a control variate, where the trajectories which relax are synchronously coupled to equilibrium ones, with initial conditions perturbed off equilibrium. The numerical results in \cref{sec:numerical_lang} show significant variance-reduction potential, suggesting this method is viable as its implementation is neither complex nor expensive; in fact, it is roughly twice the cost of a typical run due to the control system, so that the computational overhead is more than compensated by the variance reduction. For general systems, the bottleneck lies in the construction of the transformation~$\Phi_\eta$, as the PDEs \eqref{eq:varphi1_PDE} and \eqref{eq:varphi2_PDE} might not have a readily available solution for some given conjugate response~$S$ of interest. 

This works calls for several extensions. A particularly appealing one is to explore other types of couplings. For the systems we considered, synchronous coupling was largely successful in keeping the trajectories sufficiently close during the transient relaxation. For the shear viscosity example, relaxation and decoupling almost coincided, which suggests that a less dissipative system is likely to undergo decoupling significantly before convergence. Such scenarios motivate exploring more robust coupling strategies to delay decoupling, such as the ones described in~\cite{guillin2012,monmarche2024,chak2024,darshan2024}. 

Another topic of interest is to carefully compare the transient subtraction approach studied here and the TTCF method. From a theoretical perspective, the estimates obtained in our work lead to the following conclusions:
\begin{itemize}
\item Concerning the bias of the estimators, one should compare the result of \cref{cor:gk_equiv} for the transient subtraction approach, and~\eqref{eq:bias_TTCF}. The latter equality shows that the bias of the TTCF estimator does not depend on the forcing magnitude~$\eta$ and only involves a term exponentially small in the integration time; while the bias of the transient subtraction techniques involves an additional term of order~$\eta^\alpha$, which makes the latter approach unsuitable for large forcing magnitudes.
\item Concerning the variance of the estimators, one should compare the result of \cref{prop:var_ts} for the transient subtraction approach, and~\eqref{eq:asymptotic_variance_TTCF} for the recentered TTCF estimator. In both cases, the variance can be bounded uniformly in~$\eta$ (in particular in the linear response regime~$\eta \to 0$). The difference between the two estimators comes from the time dependence of the variance, which can be uniformly bounded for the transient subtraction method when the dynamics is dissipative, or growing exponentially when the dynamics is unstable (positive Lyapunov exponents); while it grows linearly in time for the recentered TTCF estimator. In practice, the integration time is of the order of the relaxation time, and so it is difficult to draw general conclusions here. We however believe that the transient subtraction technique will be relevant mostly for systems with small enough relaxation times in view of the bound provided by \cref{prop:var_ts}.
\end{itemize}
The main point in the above analysis is that the variance, which is usually the dominant error term, is uniformly bounded with respect to~$\eta$ for both methods. The difference comes from the time dependence of the variance, which we anticipate to be model dependent. This suggests numerically comparing both methods on relevant systems to assess their relative performance; for instance, polymer melts~\cite{pan2006}, confined fluids~\cite{bernardi2012, bernardi2016}, or one-dimensional atom chains (for which relaxation times can diverge as the system size increases)~\cite{lepri2003, dhar2008, iubini2020}.

\section*{Declarations}
\paragraph{Acknowledgements} This project has received funding from the European Union's Horizon 2020 research and innovation program under the Marie Sklodowska--Curie grant agreement No 945332, and from the European Research Council (ERC) under the European Union's Horizon 2020 research and innovation programme (project EMC2, grant agreement No 810367). We also acknowledge funding from the Agence Nationale de la Recherche, under grants ANR-19-CE40-0010-01 (QuAMProcs) and ANR-21-CE40-0006 (SINEQ). 

\paragraph{Conflict of interest} The authors have no relevant financial or non-financial interests to disclose.

\paragraph{Data availability} All code used to generate the presented in this work can be found in the GitHub repository \url{https://github.com/renatospacek/TransientSubtraction}. 

\printbibliography

\end{document}